\documentclass[aop,preprint]{imsart}

\RequirePackage[OT1]{fontenc}
\RequirePackage{amsmath,amsfonts,amsthm,amssymb,color,mathrsfs}
\RequirePackage[numbers]{natbib}
\RequirePackage[colorlinks,citecolor=blue,urlcolor=blue]{hyperref}
\RequirePackage{enumitem}
\RequirePackage{comment}
\RequirePackage{stmaryrd}
\usepackage{soul}
\usepackage{xcolor}

\newcommand{\bch}{\bar{\mathcal{H}}}
\newcommand{\hb}[1]{\textcolor{blue}{#1}}

\def\mr{{\mathbb  R}}

\newcommand{\beq}{\begin{equation}}
\newcommand{\eeq}{\end{equation}}

\newcommand{\E}{\mathbb E}

\newcommand{\R}{\mathbb R}

\newcommand{\ca}{\mathcal A}

\newcommand{\cac}{\mathcal C}

\newcommand{\cf}{\mathcal F}
\newcommand{\ch}{\mathcal H}

\newcommand{\ck}{\mathcal K}
\newcommand{\cl}{\mathcal L}

\newcommand{\cs}{\mathcal S}

\newcommand{\al}{\alpha}
\newcommand{\ep}{\varepsilon}

\newcommand{\ga}{\gamma}

\newcommand{\oom}{\Omega}

\newcommand{\vp}{\varphi}

\newcommand{\lp}{\big(}
\newcommand{\rp}{\big)}
\newcommand{\lc}{\big[}
\newcommand{\rc}{\big]}
\newcommand{\lcl}{\big\{}
\newcommand{\rcl}{\big\}}
\newcommand{\lln}{\big|}
\newcommand{\rrn}{\big|}

\startlocaldefs
\numberwithin{equation}{section}
 \theoremstyle{plain}
  \newtheorem{thm}{Theorem}[section]
   \theoremstyle{plain}
  
  \theoremstyle{remark}
  \newtheorem{rem}[thm]{Remark}
  \theoremstyle{definition}
  \newtheorem*{notation*}{Notation}
\newtheorem{notation}[thm]{Notation}
  \theoremstyle{plain}
  \newtheorem{prop}[thm]{Proposition}
    \theoremstyle{plain}
  \newtheorem{lem}[thm]{Lemma}
  \theoremstyle{plain}
  \newtheorem{cor}[thm]{Corollary}
  \theoremstyle{definition}
  \newtheorem{defn}[thm]{Definition}
   \theoremstyle{plain}
  
   \theoremstyle{plain}
  \newtheorem*{uha*}{Uniform Hypoellipticity Assumption}
     \theoremstyle{definition}
      \newtheorem*{uell*}{Uniform Ellipticity Assumption}
     \theoremstyle{definition}
  \newtheorem{conv}[thm]{Convention}
\endlocaldefs

\begin{document}

\begin{frontmatter}
\title{Precise Local Estimates for Differential Equations driven by Fractional Brownian Motion: Hypoelliptic Case}
\runtitle{Density Estimates for Hypoelliptic SDEs driven by fBM}

\begin{aug}
\author{\fnms{Xi} \snm{Geng}\thanksref{m1}\ead[label=e1]{xi.geng@unimelb.edu.au}},
\author{\fnms{Cheng} \snm{Ouyang}\thanksref{m2}\ead[label=e2]{couyang@math.uic.edu}}
\and
\author{\fnms{Samy} \snm{Tindel}\thanksref{m3}
\ead[label=e3]{stindel@purdue.edu}
\ead[label=u1,url]{http://www.foo.com}}

\runauthor{X. GENG, C. OUYANG AND S. TINDEL}

\affiliation{University of Melbourne\thanksmark{m1}, University of Illinois at Chicago\thanksmark{m2} and \\Purdue University\thanksmark{m3}}

\address{X. Geng\\School of Mathematics and Statistics\\University of Melbourne\\
813 Swanston Street\\
Parkville, VIC 3010\\
Australia\\
\printead{e1}\\
\phantom{E-mail:\ }}

\address{C. Ouyang\\Department of Mathematics, Statistics\\ and Computer Science\\University of Illinois at Chicago\\
851 S Morgan St\\
Chicago  IL 60607\\United States\\
\printead{e2}\\
\phantom{E-mail:\ }}

\address{S. Tindel\\Department of Mathematics\\Purdue University\\
150 N. University Street\\
West Lafayette, IN 47907-2067\\United States\\
\printead{e3}\\
\phantom{E-mail:\ }}

\end{aug}

\begin{abstract}
This article is concerned with stochastic differential equations driven by a $d$ dimensional fractional Brownian motion with Hurst parameter $H>1/4$, understood in the rough paths sense. Whenever the coefficients of the equation satisfy a uniform hypoellipticity condition, we establish a sharp local estimate on the associated control distance function and a sharp local lower estimate on the density of the solution. Our methodology relies heavily on the rough paths structure of the equation.
\end{abstract}

\begin{keyword}[class=MSC]
\kwd{60H10}
\kwd{60H07}
\kwd{60G15}
\end{keyword}

\begin{keyword}
\kwd{Rough paths}
\kwd{Malliavin calculus}
\kwd{Fractional Brownian motion}
\end{keyword}

\end{frontmatter}

\section{Introduction}

We will split our introduction into two parts. In Section \ref{sec:background}, we recall some background on the stochastic analysis of stochastic differential equations driven by a fractional Brownian motion. In Section \ref{sec: main results} we describe our main results.  Section \ref{sec: strategy} is then devoted to a brief explanation about the methodology we have used in order to obtain our main results.

\subsection{Background and motivation}\label{sec:background}
One way to envision Malliavin calculus is to see it as a geometric and analytic framework on an infinite dimensional space (namely the Wiener space) equipped with a Gaussian measure. This is already apparent in Malliavin's seminal contribution \cite{Ma78} giving a probabilistic proof of H\"ormander's theorem. The same point of view has then been  pushed forward in the celebrated series of papers by Kusuoka and Stroock~\cite{KS84,KS84b,KS87}, which set up the basis for densities and probabilities expansions for diffusion processes within this framework.

On the other hand, the original perspective of Lyons' rough path theory (cf. \cite{Lyons94, Lyons98}) is quite different. Summarizing very briefly, it asserts that a reasonable differential calculus with respect to a noisy process $X$ can be achieved as long as one can define enough iterated integrals of $X$. One of the first processes to which the theory has been successfully applied is a fractional Brownian motion $B$ with Hurst parameter $H$, on which we shall focus in the present paper. 
The process $B$ can be seen as a natural generalization of Brownian motion allowing any kind of H\"older regularity. 
We are interested in the following differential equation driven by $B$:
\begin{equation}\label{eq: hypoelliptic SDE}
\begin{cases}
dX_{t}=\sum_{\alpha=1}^{d}V_{\alpha}(X_{t})dB_{t}^{\alpha}, & 0\leq t\leq1,\\
X_{0}=x\in\mathbb{R}^{N}.
\end{cases}
\end{equation}
Here the $V_\alpha$'s are $C_b^\infty$ vector fields, and the Hurst parameter is assumed to satisfy $H>1/4$. In this setting, putting together the results contained in \cite{CQ} and \cite{Lyons98}, the stochastic differential equation \eqref{eq: hypoelliptic SDE} can be understood in the framework of rough path theory. 

With the solution of \eqref{eq: hypoelliptic SDE} in hand, a natural problem one can think of is the following: can we extend the aforementioned analytic studies on Wiener's space to the process $B$? In particular can we complete Kusuoka-Stroock's program in the fractional Brownian motion setting? This question has received a lot of attention in the recent years, and previous efforts along this line include H\"ormander type theorems for the process $X$ defined by \eqref{eq: hypoelliptic SDE} (cf. \cite{BH07,CF10,CHLT15}), some upper Gaussian bounds on the density $p(t,x,y)$ of $X_{t}$ (cf. \cite{BNOT16,GOT}), as well as Varadhan type estimates for $\log(p(t,x,y))$ in small time \cite{BOZ15}. One should stress at this point that the road from the Brownian to the fractional Brownian case is far from being trivial. This is essentially due to the lack of independence of the fBm increments and Markov property,  as well as to the technically demanding characterization of the Cameron-Martin space whenever $B$ is not a Brownian motion. We shall go back to those obstacles throughout the article.

Our contribution can be seen as a step in the direction mentioned above. More specifically, we shall obtain  a sharp local estimate on the associated control distance function and some sharp local bounds for the density of $X_{t}$ under hypoelliptic conditions on the vector fields $V_{\al}$. This will be achieved thanks to a combination of geometric and analytic tools which can also be understood as a mix of stochastic analysis and rough path theory. We describe our main results more precisely in the next subsection.

\subsection{Statement of main results}\label{sec: main results}

Let us recall that equation \eqref{eq: hypoelliptic SDE} is our main object of concern.
We are typically interested in the degenerate case where the vector fields $V=\{V_1,\ldots,V_d\}$ satisfy the \textit{uniform hypoellipticity} assumption to be defined shortly. This is a standard degenerate setting where one can expect that the solution of the SDE (\ref{eq: hypoelliptic SDE}) admits a smooth density with respect to the Lebesgue measure. As mentioned in Section \ref{sec:background}, we wish to obtain quantitative information for the density in this context.

We first formulate the uniform hypoellipticity condition which will be assumed throughout the rest of the paper. For $l\geq1,$ define $\mathcal{A}(l)$ to be the set of words over letters $\{1,\ldots,d\}$ with length at most $l$ (including the empty word), and $\mathcal{A}_1(l)\triangleq \mathcal{A}(l)\backslash\{\emptyset\}$. Denote $\mathcal{A}_1$ as the set of all non-empty words. Given a word $\alpha\in\ca_{1}$, we define the vector field $V_{[\alpha]}$ inductively by $V_{[i]}\triangleq V_i$ and $V_{[\alpha]}\triangleq[V_i,V_{[\beta]}]$ for $\alpha=(i,\beta)$ with $i$ being a letter and $\beta\in\ca_{1}$.
\begin{uha*}
The vector fields $(V_{1},\ldots,V_{d})$ are $C_{b}^{\infty}$, and there exists an integer $l_{0}\geq1$, such that
\begin{equation}\label{eq:unif-hypo-assumption}
\inf_{x\in\mathbb{R}^N}\inf_{\eta\in S^{N-1}}\lcl  \sum_{\alpha\in\mathcal{A}_1(l_0)}\langle V_{[\alpha]}(x),\eta\rangle_{\mathbb{R}^N}^{2}\rcl >0.
\end{equation}
The smallest such $l_0$ is called the hypoellipticity constant for the vector fields.
\end{uha*}


Under condition \eqref{eq:unif-hypo-assumption}, it was proved by Cass-Friz \cite{CF10} and Cass-Hairer-Litterer-Tindel \cite{CHLT15} that the solution to the SDE (\ref{eq: hypoelliptic SDE}) admits a smooth density $y\mapsto p(t,x,y)$ with respect to the Lebesgue measure on $\mathbb{R}^N$ for all $(t,x)\in (0,1]\times\mathbb{R}^N$. Our contribution aims at getting quantitative small time estimates for $p(t,x,y)$.

In order to describe our bounds on the density $p(t,x,y)$, let us recall that the small time behavior of $p(t,x,y)$ is closely related to the so-called control distance function associated with the vector fields. This fact was already revealed in the Varadhan-type asymptotics result proved by Baudoin-Ouyang-Zhang \cite{BOZ15}:
\begin{equation}\label{eq:varadhan-estimate}
\lim_{t\rightarrow0}t^{2H}\log p(t,x,y)=-\frac{1}{2}d(x,y)^{2}.
\end{equation}
The control distance function $d(x,y)$ in \eqref{eq:varadhan-estimate}, which plays a prominent role in our paper, is defined as follows. 
\begin{defn}\label{def: CM bridge}
Let $\bar{\mathcal{H}}$ be the Cameron-Martin space of the fractional Brownian motion. For any $h\in\bar{\mathcal{H}}$, denote by $\Phi_t(x;h)$ the solution to the ODE 
\begin{align}\label{eq: skeleton ODE}
dx_{t}=\sum_{\alpha=1}^{d}V_{\alpha}(x_{t})dh_{t}^{\alpha}, \quad 0\leq t\leq1,\quad\mathrm{with}\ x_{0}=x. \end{align}
In addition, for $x,y\in\mathbb{R}^N$ and $\Phi_t(x;h)$ defined as in \eqref{eq: skeleton ODE} set 
\begin{equation}\label{eq: Pi_xy}
\Pi_{x,y}\triangleq\lcl  h\in\bar{\mathcal{H}}:\Phi_{1}(x;h)=y\rcl  
\end{equation}
to be the space of Cameron-Martin paths which \textit{join} $x$ \textit{to} $y$ \textit{in the sense of differential equations}. The  \textit{control distance function} $d(x,y)=d_{H}(x,y)$ is defined by 
\[
d(x,y)\triangleq\inf\lcl  \|h\|_{\bar{{\cal H}}}:h\in\Pi_{x,y}\rcl  ,\ \ \ x,y\in\mathbb{R}^{N}.
\]
\end{defn}

According to \eqref{eq:varadhan-estimate}, one can clearly expect that the Cameron-Martin structure and the control distance function will play an important role in understanding the small time behavior of $p(t,x,y)$. However, unlike the diffusion case and due to the complexity of the Cameron-Martin structure of a fractional Brownian motion, the function $d(x,y)$ is far from being a metric and its shape is not clear. Our first main result is thus concerned with the local behavior of $d(x,y)$. It establishes a comparison between $d$ and  the Euclidian distance. 

\begin{thm}\label{thm: local comparison}
Under the same assumptions as in Theorem \ref{thm:equiv-distances}, let $l_0$ be the hypoellipticity constant in assumption \eqref{eq:unif-hypo-assumption} and $d$ be the control distance given in Definition \ref{def: CM bridge}.
There exist constants $C_1,C_2,\delta>0$, where $C_1,C_2$ depend only on $H,l_0$ and the vector fields, and where $\delta$ depends only on $l_0$ and the vector fields, such that 
\begin{equation}\label{eq:local-comparison}
C_{1}|x-y|\leq d(x,y)\leq C_{2}|x-y|^{\frac{1}{l_{0}}} \, ,
\end{equation}
for all $x,y\in \mathbb{R}^N$ with $|x-y|<\delta$.
\end{thm}

With the technique we use to prove Theorem \ref{thm: local comparison} together with some further effort, we are in fact able to establish a stronger result, namely, the local equivalence of $d$ to the sub-Riemannian distance induced by the vector fields $\{V_1,\ldots, V_d\}$. More specifically, let us write the distance given in Definition~\ref{def: CM bridge} as $d_H(x,y)$, in order to emphasize the dependence on the Hurst parameter $H$. Our second main result asserts that all the distances $d_{H}$ are locally equivalent. 

\begin{thm}\label{thm:equiv-distances}
Assume that the vector fields $(V_{1},\ldots,V_{d})$ satisfy the uniform hypoellipticity condition \eqref{eq:unif-hypo-assumption}. For $H\in (1/4,1)$, consider the distance $d_{H}$ given in Definition \ref{def: CM bridge}.
Then
for any $H_1, H_2\in (1/4,1)$, there exist  constants $C=C(H_1,H_2, V)>0$ and $\delta>0$ such that
\begin{equation}\label{eq:equiv-distances}
\frac{1}{C}d_{H_1}(x,y)\leq d_{H_2}(x,y)\leq Cd_{H_1}(x,y),
\end{equation}
for all $x,y\in\mathbb{R}^N$ with $|x-y|<\delta$.  In particular, all distances $d_H$ are locally equivalent to $d_{\textsc{BM}}\equiv d_{1/2}$, where $d_{\textsc{BM}}$ stands for the controlling distance of the system \eqref{eq: skeleton ODE} driven by a Brownian motion, i.e. the sub-Riemannian distance induced by the vector fields $\{V_1,...,V_d\}$.
\end{thm}

The local equivalence of distances stated in Theorem \ref{thm:equiv-distances} plays a crucial role in our later analysis of the density $p(t,x,y)$, in particular, in the proof of Theorem \ref{thm: local lower estimate} below (e.g. see the proof of Lemma \ref{lem: estimate vol m_xx} below). Moreover, we believe Theorem \ref{thm:equiv-distances} and the tools we developed in order to prove this theorem may be of independent interest and have other applications as well.

\begin{rem}\label{rmk:density-signature}
In the special case when \eqref{eq: hypoelliptic SDE} reads as 
$dX_t=X_t\otimes dB_t,$
that is, when $X_t$ is the truncated signature of $B$ up to order $l>0$, it is proved in \cite{BFO19} that all $d_H(x,y)$ are globally equivalent. The proof crucially depends on the fact that the signature of $B$ is homogeneous with respect to the dilation operator on $G^{(l)}(\mr^d)$, the free nilpotent Lip group over $\mr^d$ of order $l$. In the current general nonlinear case, the local equivalence is much more technically challenging.  In addition, we think that the global equivalence of the distances $d_{H}$ does not hold.
\end{rem}

Our third main result asserts that the density $p(t,x,y)$ of $X_{t}$ is strictly positive everywhere whenever $t>0$. It generalizes for the first time the result of \cite[Theorem 1.4]{BNOT16} to a general hypoelliptic case, by affirming that Hypothesis 1.2 in that theorem is always verified under our assumption \eqref{eq:unif-hypo-assumption}. Recall that a \textit{distribution} over a differentiable manifold is a smooth choice of subspace of the tangent space at every point with constant dimension.  

\begin{thm}\label{prop: positivity-intro}
Let $\{V_{1},\ldots,V_{d}\}$ be a family of $C_{b}^{\infty}$-vector
fields on $\mathbb{R}^{N},$ which span a distribution over $\mathbb{R}^N$ and satisfy the uniform hypoellipticity assumption \eqref{eq:unif-hypo-assumption}.  Let $X_{t}^{x}$
be the solution to the stochastic differential equation~\eqref{eq: hypoelliptic SDE},
where $B_{t}$ is a $d$-dimensional fractional Brownian motion with
Hurst parameter $H>1/4$. Then for each $t\in(0,1]$, the density
of $X_{t}$ is everywhere strictly positive.
\end{thm}

As we will see in Section \ref{section: step 1}, the proof of the above result is  based on finite dimensional geometric arguments such as the classical Sard theorem, as well as a general positivity criteria for densities on the Wiener space. 
After our preprint was released, further results on strict positivity were obtained very recently in \cite{IP} using the notion of $\ck$-regularity. However, the latter reference focuses exclusively on the strict positivity of density for Gaussian rough differential equations. In contrast, the proof of Theorem~\ref{prop: positivity-intro} comes as a natural and short byproduct of our global approach.

Let us now turn to a description of our last main result. It establishes a sharp local lower estimate for the density function $p(t,x,y)$ of the solution to the SDE (\ref{eq: hypoelliptic SDE}) in small time. 

\begin{thm}\label{thm: local lower estimate}
Under the uniform hypoellipticity assumption (\ref{eq:unif-hypo-assumption}), let $p(t,x,y)$ be the density of the random variable $X_{t}$ defined by equation \eqref{eq: hypoelliptic SDE}.
There exist some constants $C,\tau>0$ depending only on $H,l_0$ and the vector fields $V_{\al}$, such that 
\begin{equation}\label{eq:local-density-bound}
p(t,x,y)\geq\frac{C}{|B_{d}(x,t^{H})|},
\end{equation}
for all $(t,x,y)\in(0,1]\times\mathbb{R}^N\times\mathbb{R}^N$ satisfying the following local condition involving the distance $d$ introduced in Definition \ref{def: CM bridge}:
\begin{equation*}
d(x,y)\leq t^H,
\quad\text{and}\quad
t<\tau .
\end{equation*}
In relation \eqref{eq:local-density-bound}, $B_{d}(x,t^{H})\triangleq\{ z\in\mathbb{R}^{N}:d(x,z)<t^{H}\}$ denotes the ball with respect to the distance $d$ and $|\cdot|$ stands for the Lebesgue measure.
\end{thm}

The sharpness of Theorem \ref{thm: local lower estimate} can be seen from the fractional Brownian motion case, i.e. when $N=d$ and $V=\mathrm{Id}$. {We also point out that, in terms of uniform local lower estimates, the cone $d(x,y)\leq t^H$ is the regime where the estimate appears to be interesting as it controls the rate of explosion (singularity) of the density as $t\rightarrow 0^+$}. Finally, let us stress again that, the developments of the above main results are intrinsically connected. Indeed, as we will see, the technique we use to prove Theorem \ref{thm: local comparison} will be an essential ingredient for establishing Theorem \ref{thm:equiv-distances}. In addition, Theorem \ref{thm:equiv-distances} and Theorem \ref{prop: positivity-intro} provide two essential ingredients towards the proof of Theorem \ref{thm: local lower estimate}.

\subsection{Strategy and outlook}\label{sec: strategy}

Let us say a few words about the methodology we have used in order to obtain our main results. Although we will describe our overall strategy with more details in Section \ref{sec: local lower bound}, let us  mention here that it
 is based on the reduction of the problem to a finite dimensional one, plus some geometric type arguments.  Let us also highlight the fact that a much simpler strategy can be used for the elliptic case, as explained in our companion paper~\cite{GOT20}.

More specifically, the key point in our proofs is that the solution $X_{t}$ to \eqref{eq: hypoelliptic SDE} can be approximated by a simple enough function of the so-called truncated signature of order $l$ for the fractional Brownian motion $B$. This object is formally defined,  for a given $l\ge 1$, as the following $\oplus_{k=0}^{l}(\R^{d})^{\otimes k}$-valued process:
\begin{equation*}
\Gamma_t
=
1+\sum_{k=1}^{l}\int_{0<t_1<\cdots<t_k<t}dB_{t_1}\otimes \cdots\otimes dB_{t_k},
\end{equation*}
and  it enjoys some convenient algebraic and analytic properties. The truncated signature is the main building block of the rough path theory (see e.g \cite{Lyons98}), and was also used in \cite{KS87} in a Malliavin calculus context. Part of our challenge in the current contribution is to combine the properties of the process $\Gamma$, together with the Cameron-Martin space structure related to the fractional Brownian motion $B$, in order to achieve sharp estimates for the control distance function as well as the density of solution.

As mentioned above, the truncated signature gives rise to a $l$-th order local approximation of $X_{t}$ in a neighborhood of its initial condition $x$. Namely, if $V_{(i_{1},\ldots,i_{k})}$ denotes the composition of $V_{i_{1}}$ up to $V_{i_{k}}$ considered as differential operators
and if we set 
\begin{equation}\label{eq:def-Fl-intro}
F_{l}(\Gamma_{t},x)\triangleq
\sum_{k=1}^{l} \sum_{i_{1},\ldots,i_{k}=1}^{d}
V_{(i_{1},\ldots,i_{k})}(x)
\int_{0<t_1<\cdots<t_k<t}dB_{t_1}^{i_{1}} \cdots dB_{t_k}^{i_{k}},
\end{equation}
then classical rough paths considerations assert that $F_{l}(\Gamma_{t},x)$ is an approximation of order $t^{Hl}$ of $X_{t}$ for small $t$. In the sequel we will heavily rely on some non degeneracy properties of $F_{l}$ derived from the uniform hypoelliptic assumption \eqref{eq:unif-hypo-assumption}, in order to get the following information:

\vspace{2mm}

\noindent (i) One can construct a path $h$ in the Cameron-Martin space of $B$ which joins $x$ and any point $y$ in a small enough neighborhood of $x$. This task is carried out thanks to a complex iteration procedure, whose building block is the non-degeneracy of the function $F_{l}$. It is detailed in Section \ref{sec:proof-thm-12}. In this context, observe that the computation of the Cameron-Martin norm of $h$ also requires a substantial effort. This will be the key step in order to prove Theorem \ref{thm: local comparison} and Theorem \ref{thm:equiv-distances}  concerning the distance $d$ given in Definition \ref{def: CM bridge}.\\
(ii) The proof of the lower bound given in Theorem \ref{thm: local lower estimate} also hinges heavily on the approximation $F_{l}$ given by \eqref{eq:def-Fl-intro}. Indeed, the preliminary results about the density of $\Gamma_{t}$, combined with the non-degeneracy of $F_{l}$, yield good properties for the density of $F_{l}(\Gamma_{t},x)$. One is then left with the task of showing that $F_{l}(\Gamma_{t},x)$ approximates $X_{t}$ properly at the density level.

\vspace{2mm}

In conclusion, although the steps performed in the remainder of the article might look technically and computationally involved, they rely on a natural combination of analytic and geometric bricks as well as a reduction to a finite dimensional problem. Let us also highlight the fact that our next challenge is to iterate the local estimates presented here in order to get Gaussian type lower bounds for the density $p(t,x,y)$ of $X_{t}$. This causes some further complications due to  the complex (non Markovian) dependence structure for the increments of the fractional Brownian motion $B$. We defer this important project to a future publication.
Eventually, we mention that this note is an abridged and self-contained version of our original draft on the same topic. For further computational details, the reader is referred to \cite{GOT19}.

\begin{rem}
As the reader might have observed, our equation~\eqref{eq: hypoelliptic SDE} does not have a drift component. If one wishes to add a drift $V_{0}(X_{t}) \, dt$ to the equation, it first means that the approximation function $F_{l}$ in~\eqref{eq:def-Fl-intro} should include mixed integrals involving both $dt$ and $dB_{t}$ differentials. This could possibly be achieved thanks to tree type notation, similarly to what is done e.g in~\cite{NNRT}. Once the approximation map $F_{l}$ is constructed, some subtle effects coming from the drift term are observed in the hypoelliptic case and should be taken into account in our density estimates. We refer to \cite{BW} for an analysis of the small noise asymptotics for the probability of sets $A$ which are \emph{non-horizontally accessible}, meaning that the drift $V_{0}$ is needed to reach those sets. Our density bounds should go along the same lines in the hypoelliptic case with drift. As in Kusuoka and Stroock \cite{KS87}, we have refrained to incorporate those elaborate developments in the current paper for the sake of conciseness and simplicity. However, they are certainly worth considering for a future contribution.
\end{rem}

\medskip

\noindent
\textbf{Organization of the present paper.} In Section \ref{sec:prelim}, we present some basic notions from the analysis of fractional Brownian motion and rough path theory.  In Section \ref{sec: control distance}, we develop the proofs of Theorem \ref{thm: local comparison} and Theorem \ref{thm:equiv-distances} concerning the control distance function. In Section~\ref{sec: local lower bound}, we develop the proof of Theorem \ref{thm: local lower estimate} concerning the density of solution. Theorem~\ref{prop: positivity-intro} is proved in the first key step towards proving Theorem \ref{thm: local lower estimate} in Section \ref{section: step 1}.

\begin{notation*}
Throughout the rest of this paper, we will use "$\mathrm{Letter}_{\mathrm{subscript}}$" to denote constants whose value depend only on the objects specified in the "subscript" and may differ from line to line. Unless otherwise stated, a constant will implicitly depend on $H,V,l_0$. We will always omit the dependence on dimension. 
\end{notation*}

\section{Preliminary results}\label{sec:prelim}

This section is devoted to some preliminary notions on the stochastic analysis of fractional Brownian motion. We  also recall some basic facts about rough path solutions to noisy differential equations.

\subsection{Stochastic analysis of fractional Brownian motion}\label{sec: prel.}

Let us start by recalling the definition of fractional Brownian motion.

\begin{defn}\label{def:fbm}
A $d$-dimensional \textit{fractional Brownian motion} with Hurst parameter $H\in(0,1)$ is an $\mathbb{R}^d$-valued continuous centered Gaussian process $B_t=(B_t^1,\ldots,B_t^d)$ whose covariance structure is given by 
\beq\label{eq:cov-fBm}
\mathbb{E}[B_{s}^{i}B_{t}^{j}]=\frac{1}{2}\lp s^{2H}+t^{2H}-|s-t|^{2H}\rp \delta_{ij}
\triangleq R(s,t) \delta_{ij}.
\eeq
\end{defn}

This process is defined and analyzed in numerous articles (cf. \cite{DU97,Nualart06,PT00} for instance), to which we refer for further details. We always assume that the fractional Brownian motion $B$ is defined on a complete probability space $(\Omega, \cf, \mathbb{P})$ where $\cf$ is generated by $\{B_{t}: t\in[0,T]\}$.
In this section, we focus on basic stochastic analysis notions which will be used in the sequel.

In order to introduce the Hilbert spaces which will feature in the sequel, we consider a one dimensional fractional Brownian motion $\{B_t:0\leq t\leq 1\}$ with Hurst parameter $H\in(0,1)$ defined on $(\Omega,\cf,\mathbb{P})$.   The discussion here can be easily adapted to the multidimensional setting with arbitrary time horizon $[0,T]$. 
Let ${\cal C}_{1}$ be the associated first order Wiener chaos,
i.e. ${\cal C}_{1}\triangleq\overline{\mathrm{Span}\{B_{t}:0\leq t\leq1\}}\ {\rm in}\ L^{2}(\Omega,\mathbb{P})$.
We will frequently identify a Hilbert space with its dual in the canonical way without  further mentioning. 

\begin{defn}\label{def:bar-H} Let $W$ be the space of continuous paths $w:[0,1]\rightarrow\mathbb{R}^{1}$
with $w_{0}=0.$
Define $\bar{{\cal H}}$ to be the space of elements $h\in W$ that
can be written as 
\beq\label{eq:def-h-in-CM}
h_{t}=\mathbb{E}[B_{t}Z],\ \ \ 0\leq t\leq1,
\eeq
where $Z\in{\cal C}_{1}.$ We equip $\bar{{\cal H}}$ with an inner product structure given by
\[
\langle h,h'\rangle_{\bar{{\cal H}}}\triangleq\mathbb{E}[Z\cdot Z'],\ \ \ h,h'\in\bar{{\cal H}},
\]
whenever $h,h'$ are defined by \eqref{eq:def-h-in-CM} for two random variables $Z,Z'\in{\cal C}_{1}$.
The Hilbert space $(\bar{\mathcal{H}},\langle\cdot,\cdot\rangle_{\bar{\mathcal{H}}})$ is called the \textit{Cameron-Martin
subspace} of the fractional Brownian motion. 
\end{defn}

One of the advantages of working with fractional Brownian motion is that a convenient analytic description of $\bar{\mathcal{H}}$ in terms of fractional calculus is available.  
We refer to the aforementioned references \cite{DU97,Nualart06,PT00} for the definition of fractional derivatives and some further characterizations of $\bch$. We mention one useful fact here: there is an isomorphism $K$ between $L^{2}([0,1])$
and $I_{0+}^{H+1/2}(L^{2}([0,1]))$ defined by
\begin{equation}\label{eq: analytic expression of K}
K\varphi\triangleq\begin{cases}
C_{H}\cdot I_{0^{+}}^{1}\lp t^{H-\frac{1}{2}}\cdot I_{0^{+}}^{H-\frac{1}{2}}\lp s^{\frac{1}{2}-H}\varphi(s)\rp (t)\rp , & H>\frac{1}{2};\\
C_{H}\cdot I_{0^{+}}^{2H}\lp t^{\frac{1}{2}-H}\cdot I_{0^{+}}^{\frac{1}{2}-H}\lp s^{H-\frac{1}{2}}\varphi(s)\rp (t)\rp , & H\leq\frac{1}{2},
\end{cases}
\end{equation}
where $C_{H}$ is a universal constant depending only on $H$. Then the space  $\bar{{\cal H}}$ can be identified with $I_{0^{+}}^{H+1/2}(L^{2}([0,1])),$ and 
the Cameron-Martin norm is given by 
\begin{equation}
\|h\|_{\bar{{\cal H}}}=\|K^{-1}h\|_{L^{2}([0,1])}.\label{eq: inner product in terms of fractional integrals}
\end{equation}

Next we mention a variational embedding theorem for the Cameron-Martin subspace $\bar{\mathcal{H}}$ which will be used in a crucial way. The case when $H>1/2$ is a simple exercise, while the case when $H\leq 1/2$ was treated in \cite{FV06}. From a pathwise viewpoint, this allows us to integrate a fractional Brownian path against a Cameron-Martin path or vice versa (cf. \cite{Young36}), and to make sense of ordinary differential equations driven by a Cameron-Martin path (cf.~\cite{Lyons94}).

\begin{prop}\label{prop: variational embedding}If $H>\frac{1}{2}$, then
$\bar{\mathcal{H}}\subseteq C_{0}^{H}([0,1];\mathbb{R}^{d})$, the space of H-H\"older continuous paths. If $H\leq\frac{1}{2}$,
then for any $q>\lp H+1/2\rp ^{-1}$, we have 
$\bar{\mathcal{H}}\subseteq C_{0}^{q\text{-}\mathrm{var}}([0,1];\mathbb{R}^{d})$, the space of continuous paths with finite $q$-variation. In addition, the above inclusions are continuous embeddings.
\end{prop}

For the sake of conciseness, we only recall a few notation on Malliavin calculus which feature prominently in our future considerations. The reader is referred to~\cite{Nualart06} for a thorough introduction to Malliavin calculus techniques. 
We use $\mathbf{D}_h$ to denote the Malliavin derivative along a direction $h\in\bar{{\cal H}}$. Respectively, higher order derivatives are denoted by $\mathbf{D}^k_{h_1,\ldots,h_k}$. The corresponding Sobolev spaces are written as $\mathbb{D}^{k,p}(\bar{\ch})$, and we also set $\mathbb{D}^{\infty}(\bar{\ch})=\cap_{p \geq 1} \cap_{k\geq 1} \mathbb{D}^{k,p}(\bar{\ch})$.

\begin{defn}\label{non-deg}
Let $F=(F^1,\ldots , F^n)$ be a random vector whose components are in $\mathbb{D}^\infty(\bar{\ch})$. Define the \textit{Malliavin covariance matrix} of $F$ by
\begin{equation} \label{malmat}
\gamma_F=(\langle \mathbf{D}F^i, \mathbf{D}F^j\rangle_{\bar{\ch}})_{1\leq i,j\leq n}.
\end{equation}
Then $F$ is said to be  {\it non-degenerate} if $\gamma_F$ is invertible $a.s.$ and
$$(\det \gamma_F)^{-1}\in \cap_{p\geq1}L^p(\Omega).$$
\end{defn}
It is a fundamental result in Malliavin calculus that the law of a non-degenerate random vector $F$ admits a smooth density with respect to the Lebesgue measure on $\mathbb{R}^n$ (cf. \cite[Theorem 2.1.4]{Nualart06}).

\subsection{Free nilpotent groups}\label{sec:free-nilpotent}

Next, we introduce some basic algebraic structure that plays an essential role in the rough path analysis of  equation \eqref{eq: hypoelliptic SDE}. The reader is referred to \cite{FV10} for a more systematic presentation. We start with the following notation. 
\begin{notation}\label{not:free-algebra}
The truncated tensor algebra of order $l$ over  $\R^{d}$ is denoted by $T^{(l)}$. Under addition and tensor product, $(T^{(l)},+,\otimes)$ is an associative algebra. The set of homogeneous Lie polynomials of degree $k$ is denoted as $\cl_{k}$, and the free nilpotent Lie algebra of order $l$ is denoted as $\mathfrak{g}^{(l)}$. The exponential map on $T^{(l)}$ is defined by
\beq\label{eq:def-exp-on-T-l}
\exp({a})\triangleq\sum_{k=0}^{\infty}\frac{1}{k!} \, {a}^{\otimes k}\in T^{(l)},
\eeq
where the sum is indeed locally finite and hence well-defined. $T^{(l)}$ is equipped with a natural inner product induced by the Euclidean structure on $\mathbb{R}^d$ and the Hilbert-Schmidt tensor norm on each of the tensor products. To be more precise, the inner product on $(\mathbb{R}^d)^{\otimes k}$ is induced by 
\[
\langle v_{1}\otimes\cdots\otimes v_{k},w_{1}\otimes\cdots\otimes w_{k}\rangle_{(\mathbb{R}^{d})^{\otimes k}}\triangleq\langle v_{1},w_{1}\rangle_{\mathbb{R}^{d}}\cdots\langle v_{k},w_{k}\rangle_{\mathbb{R}^{d}}.
\]The components $(\mathbb{R}^d)^{\otimes k}$ and $(\mathbb{R}^d)^{\otimes m}$ ($k\neq m$) are assumed to be orthogonal.
\end{notation}

The following algebraic structure is critical in rough path theory.
\begin{defn}\label{def:free-group}
The \textit{free nilpotent Lie group} $G^{(l)}$ of order $l$ is defined by 
$$G^{(l)}\triangleq\exp(\mathfrak{g}^{(l)})\subseteq T^{(l)}.
$$ 
The exponential function is a diffeomorphism under which $\mathfrak{g}^{(l)}$ in Notation~\ref{not:free-algebra} is the Lie algebra of $G^{(l)}$.
\end{defn}

It will be useful in the sequel to have some basis available for the algebras introduced above. Recall that $\mathcal{A}(l)$ (respectively, $\mathcal{A}_1(l)$) denotes the set of words (respectively, non-empty words) over $\{1,\cdots,d\}$ of length at most $l$. We set $\mathrm{e}_{(\emptyset)}\triangleq1$, and for each word $\alpha=(i_{1},\ldots,i_{r})\in\ca_{1}(l)$, we set
\beq\label{eq:def-basis-algebras}
{\rm e}_{(\alpha)}\triangleq{\rm e}_{i_{1}}\otimes\cdots\otimes{\rm e}_{i_{r}},
\quad\text{and}\quad
{\rm e}_{[\alpha]}\triangleq[{\rm e}_{i_{1}},\cdots[{\rm e}_{i_{r-2}},[{\rm e}_{i_{r-1}},{\rm e}_{i_{r}}]]],
\eeq
where $\{{\rm e}_{1},\ldots,{\rm e}_{d}\}$ denotes the canonical basis of $\mathbb{R}^{d}
$. 
Then $\{{\rm e}_{(\alpha)}:\alpha\in{\cal A}(l)\}$ forms an orthonormal
basis of $T^{(l)}$ under the Hilbert-Schmidt tensor norm. In addition, we also have $\mathfrak{g}^{(l)}={\rm Span}\{{\rm e}_{[\alpha]}:\alpha\in{\cal A}_{1}(l)\}$.

As a closed subspace, $\mathfrak{g}^{(l)}$ inherits a canonical Hilbert structure from $T^{(l)}$ which makes it into a flat Riemannian manifold. The associated volume measure $du$ (the Lebesgue measure) on $\mathfrak{g}^{(l)}$ is left invariant with respect to the product induced from the group structure on $G^{(l)}$ through the exponential diffeomorphism. In addition, for each $\lambda>0$, there is a dilation operation $\delta_\lambda: T^{(l)}\rightarrow T^{(l)}$ induced by $\delta_\lambda(a)\triangleq \lambda^k a$ if $a\in(\mathbb{R}^d)^{\otimes k}$, which satisfies the relation 
$\delta_\lambda\circ\exp=\exp\circ\,\delta_\lambda$ when restricted on $\mathfrak{g}^{(l)}$. Thanks to the fact that $\delta_\lambda(a)= \lambda^k a$ for any $a\in(\R^{d})^{\otimes k}$ and recalling that $\mathcal{L}_k$ is introduced in Notation~\ref{not:free-algebra},
one can easily show that 
\begin{equation}\label{eq:dilation-lebesgue-on-cal-G}
du\circ\delta_\lambda^{-1}=\lambda^{-\nu}du,
\quad\text{where}\quad
\nu\triangleq\sum_{k=1}^{l}k\dim (\mathcal{L}_k).
\end{equation}

We always fix the Euclidean norm on $\mathbb{R}^d$ in the remainder of the paper. As far as the free nilpotent group $G^{(l)}$ is concerned, there are several useful metric structures. Among them we will use an extrinsic metric $\rho_{\textsc{HS}}$ which can be defined easily due to the fact that $G^{(l)}$ is a subspace of $T^{(l)}$. Namely for $g_{1},g_{2}\in G^{(l)}$ we set:
\beq\label{eq:def-HS-distance}
\rho_{\textsc{HS}}(g_{1},g_{2})\triangleq\|g_{2}-g_{1}\|_{\textsc{HS}},\ \ \ g_{1},g_{2}\in G^{(l)},
\eeq
where the right hand side is induced from the Hilbert-Schmidt norm on $T^{(l)}$.

\subsection{Path signatures and the fractional Brownian rough path}\label{sec:signature}

The stochastic differential equation \eqref{eq: hypoelliptic SDE} driven by a fractional Brownian motion $B$ is standardly solved in the rough paths sense. In this section we recall some basic facts about this notion of solution. We will also give some elements of rough paths expansions, which are at the heart of our methodology in the present paper. 

The link between free nilpotent groups and noisy equations like \eqref{eq: hypoelliptic SDE} is made through the notion of signature. Recall that a continuous map
$\mathbf{x}:\{(s,t)\in[0,1]^{2}: \, s\le t\}\rightarrow T^{(l)}$ is called a
\textit{multiplicative functional} if for $s<u<t$ one has $\mathbf{x}_{s,t}
=\mathbf{x}_{s,u}\otimes\mathbf{x}_{u,t}$. A particular occurrence of this kind of map is given when one considers a path $w$ with finite
variation and sets for $s\le t$,
\begin{equation}\label{eq:def-iterated-intg}
\mathbf{w}_{s,t}^{n}=
 \int_{s<u_{1}<\cdots<u_{n}<t}dw_{u_{1}}\otimes\cdots\otimes dw_{u_{n}} .
\end{equation}
Then the so-called \textit{truncated signature path} of order $l$ associated with $w$ is the functional $S_{l}(w)_{\cdot,\cdot}: \{(s,t)\in[0,1]^{2}: s\le t\} \rightarrow T^{(l)}$ defined by
\begin{equation}\label{eq:signature-smooth-x}
S_{l}(w)_{s,t}:=1+\sum_{n=1}^{l}\mathbf{w}_{s,t}^{n}.
\end{equation}
It can be shown that the functional $S_{l}(w)_{\cdot,\cdot}$ is multiplicative and takes values in the free nilpotent group $G^{(l)}$. The \textit{truncated signature} of order $l$ for $w$ is the tensor element $S_l(w)_{0,1}\in G^{(l)}$. It is simply denoted as $S_l(w)$.

A rough path can be seen as a generalization of the truncated signature path \eqref{eq:signature-smooth-x} to the non-smooth situation. Specifically, the definition of H\"older rough paths can be summarized as follows.
 \begin{defn}\label{def:RP}
Let $\ga\in(0,1)$. The space of \textit{weakly geometric $\ga$-H\"older rough paths} 
is the set of multiplicative functionals $\mathbf{x}:\{(s,t)\in[0,1]^{2}: s\le t\}\rightarrow G^{[1/\ga]}$ such that 
\begin{equation}
\|\mathbf{x}\|_{\ga;\textsc{HS}}
\triangleq
\sum _{i=1}^{[1/\ga]} \sup_{0\le s < t \le 1} \frac{\|\mathbf{x}^i_{s,t}\|_{\textsc{HS}}}{|t-s|^{i\ga}} < \infty.
\end{equation}
\end{defn}
An important subclass of weakly geometric $\ga$-H\"older rough paths is the set of \textit{geometric} $\ga$-H\"older rough paths. These are multiplicative paths $\mathbf{x}$ with values in $G^{\lfloor
1/\ga\rfloor}$ such that $\|\mathbf{x}\|_{\ga;\textsc{HS}}$ is finite and such that there exists a sequence 
$\{  x_{\ep}: \ep>0 \} $ with $x_{\ep}\in C^{\infty}([0,1];\mathbb{R}^{d})$ satisfying 
\begin{equation}\label{eq:closure-smooth-path}
\lim_{\ep\to 0} 
\|\mathbf{x} - S_{[1/\ga]}(x_{\ep})\|_{\ga;\textsc{HS}}
 =0.
\end{equation}


The notion of signature allows us to define a more intrinsic notion of distance on the free group $G^{(l)}$. This metric, known as the \textit{Carnot-Caratheodory metric}, is defined  by 
\[
\rho_{\textsc{CC}}(g_{1},g_{2})\triangleq\|g_{1}^{-1}\otimes g_{2}\|_{\textsc{CC}},\ \ \ g_{1},g_{2}\in G^{(l)},
\]
where the CC-norm $\|\cdot\|_{\textsc{CC}}$ is defined by  
\beq\label{eq:def-norm-CC}
\|g\|_{\textsc{CC}}\triangleq\inf\lcl  \|w\|_{1-{\rm var}}:w\in C^{1-{\rm var}}([0,1];\mathbb{R}^{d})\ {\rm and}\ S_{l}(w)=g\rcl . 
\eeq
It can be shown that the infimum in \eqref{eq:def-norm-CC} is attainable. 

\begin{rem}\label{connecting by smooth path}
It is well-known that for any $g\in G^{(l)}$, one can find a piecewise linear path $w$ such that $S_l(w)=g$ (cf. \cite{FV10} for instance). In addition, a simple reparametrization allows one to take $w$ to be smooth with derivative compactly supported in $(0,1)$. 
\end{rem}

The HS and CC metrics are equivalent as seen from the following so-called \textit{ball-box estimate} (cf. \cite[Proposition 7.49]{FV10}).
\begin{prop}\label{prop: ball-box estimate}
Let $\rho_{\textsc{HS}}$ and $\rho_{\textsc{CC}}$ be the distances on $G^{(l)}$ respectively defined by~\eqref{eq:def-HS-distance} and \eqref{eq:def-norm-CC}.
For each $l\geq1$, there exists a constant $C=C_{l}>0$, such that 
\begin{equation}\label{eq:bnd-CC-to-HS}
\rho_{\textsc{CC}}(g_{1},g_{2})\leq C\max\lcl  \rho_{\textsc{HS}}(g_{1},g_{2}),\rho_{\textsc{HS}}(g_{1},g_{2})^{\frac{1}{l}}\cdot\max\{ 1,\|g_{1}\|_{\textsc{CC}}^{1-\frac{1}{l}}\} \rcl  
\end{equation}
and
\[
\rho_{\textsc{HS}}(g_{1},g_{2})\leq C\max\lcl  \rho_{\textsc{CC}}(g_{1},g_{2})^{l},\rho_{\textsc{CC}}(g_{1},g_{2})\cdot\max\{ 1,\|g_{1}\|_{\textsc{CC}}^{l-1}\} \rcl  
\]
for all $g_{1},g_{2}\in G^{(l)}$. In particular, 
\[
\|g\|_{\textsc{CC}}\leq1\implies\|g-{\bf 1}\|_{\textsc{HS}}\leq C\|g\|_{\textsc{CC}}
\]
and
\[
\|g-{\bf 1}\|_{\textsc{HS}}\leq1\implies\|g\|_{\textsc{CC}}\leq C\|g-{\bf 1}\|_{\textsc{HS}}^{\frac{1}{l}}.
\]
\end{prop}

One of the main application of rough path theory is to extend most stochastic calculus tools to a large class of Gaussian processes. The following result, borrowed from \cite{CQ,FV06}, establishes this link for fractional Brownian motion. 

\begin{prop}\label{prop:fbm-rough-path}
Let $B$ be a fractional Brownian motion with Hurst parameter  $H>1/4$. Then $B$ admits a lift $\mathbf{B}$ as a geometric rough path of order $[1/\ga]$ for any $\ga<H$.
\end{prop}

Let us now turn to the definition of rough differential equations. There are several equivalent ways to introduce this notion, among which we will choose to work with Taylor type expansions, since they are more consistent with our later developments. To this aim, let us first consider a bounded variation path $w$ and the following ordinary differential equation driven by $w$:
\begin{equation}\label{eq: ode}
dx_t=\sum_{\alpha=1}^d V_\alpha(x_t) \, dw_t^\alpha,
\end{equation}
where the $V_\alpha$'s are $C_b^\infty$ vector fields.
For any given word $\alpha=(i_1,\ldots,i_r)$ over the letters $\{1,\ldots,d\}$,  define the vector field $V_{(\alpha)}\triangleq(V_{i_{1}}\cdots(V_{i_{r-2}}(V_{i_{r-1}}V_{i_{r}})))$, 
where we have identified a vector field with a differential operator, so that $V_iV_j$ means differentiating $V_j$ along direction $V_i$. 
Classically, a \textit{formal Taylor expansion} of the solution $x_t$ to \eqref{eq: ode} is  given by
\begin{equation}\label{eq: formal Taylor expansion}
x_{s,t}\sim\sum_{k=1}^{\infty}\sum_{i_{1},\ldots,i_{k}=1}^{d}V_{(i_{1},\ldots,i_{k})}(x_{s})
\int_{s<u_{1}<\cdots<u_{k}<t}dw_{u_{1}}^{i_{1}}\cdots dw_{u_{k}}^{i_{k}},
\end{equation}
where we have set $x_{s,t}=x_{t}-x_{s}$.  
This expansion can be rephrased in more geometrical terms. Specifically, we define the following Taylor approximation function on $\mathfrak{g}^{(l)}$. Recall that the sets of words $\ca(l),\ca_{1}(l)$ are introduced at the beginning of Section \ref{sec: main results}.

\begin{defn}\label{def: Taylor approximation function}
For each $l\geq1$, we define the \textit{Taylor approximation function} $F_l: \mathfrak{g}^{(l)}\times\mathbb{R}^N\rightarrow\mathbb{R}^N$ of order $l$ associated with the ODE (\ref{eq: ode}) by
\[
F_{l}(u,x)\triangleq\sum_{\alpha\in{\cal A}_{1}(l)}V_{(\alpha)}(x)\cdot (\exp u)^\alpha,\ \ \ (u,x)\in\mathfrak{g}^{(l)}\times\mathbb{R}^{N},
\]
where the exponential function is defined on $T^{(l)}$ by \eqref{eq:def-exp-on-T-l} and $(\exp(u))^\alpha$ is the coefficient of $\exp(u)$ with respect to the tensor basis element $\rm e_{(\alpha)}$ introduced in \eqref{eq:def-basis-algebras}. We also say that $u\in\mathfrak{g}^{(l)}$ \textit{joins} $x$ \textit{to} $y$ \textit{in the sense of Taylor approximation} if $y=x+F_l(u,x)$.
\end{defn}

With Definition \ref{def: Taylor approximation function} in hand, we can recast the formal expansion \eqref{eq: formal Taylor expansion} (truncated at an arbitrary degree $l$) in the following way:
\begin{equation}\label{eq: formal Taylor expansion 2}
x_{s,t}\sim F_{l}\lp \log\lp S_l(w)_{s,t} \rp,  x_{s}  \rp ,
\end{equation}
where the function $\log$ is the inverse of the exponential map for $G^{(l)}$, and $S_l(w)_{s,t}$ is the truncated signature path of $w$ defined by \eqref{eq:signature-smooth-x}. In order to define rough differential equations, a natural idea is to extend this approximation scheme to rough paths. We get a definition which is stated below in the fractional Brownian motion case.

\begin{defn}\label{def:solution-rde}
Let $B$ be a fractional Brownian motion with Hurst parameter  $H>1/4$, and consider its rough path lift $\mathbf{B}$ as in Proposition \ref{prop:fbm-rough-path}. Let $\{V_\alpha:1\le \al\le d\}$ be a family of $C_b^\infty$ vector fields on $\R^{N}$. We say that $X$ is a \textit{solution to the rough differential equation}~\eqref{eq: hypoelliptic SDE} if for all $(s,t)\in[0,1]^{2}$ such that $s<t$ we have
\begin{equation}\label{eq:def-rough-eq}
X_{s,t}
=
F_{[1/\gamma]-1}\lp \log\lp S_{[1/\gamma]-1}(\mathbf{B})_{s,t} \rp,  X_{s}  \rp + R_{s,t} \, ,
\end{equation}
where $R_{s,t}$ is an $\R^{N}$-valued remainder such that there exists $\ep>0$ satisfying
\begin{equation*}
\sup_{0\le s < t \le 1} \frac{|R_{s,t}|}{|t-s|^{1+\ep}} < \infty.
\end{equation*}
\end{defn}

Roughly speaking, Definition \ref{def:solution-rde} says that the expansion of the solution $X$ to a rough differential equation should coincide with \eqref{eq: formal Taylor expansion} up to a remainder with H\"older regularity greater than 1. This approach goes back to Davie \cite{Da07}, and it can be shown to coincide with more classical notions of solutions. The following existence and uniqueness result is fundamental in rough path theory.
\begin{prop}
Under the same conditions as in Definition \ref{def:solution-rde}, there exists a unique solution to equation \eqref{eq: hypoelliptic SDE} understood in the sense of~\eqref{eq:def-rough-eq}.
\end{prop}

\section{Local estimate for the control distance function}\label{sec: control distance} 

In this section, we develop the proofs of Theorem \ref{thm: local comparison} and Theorem \ref{thm:equiv-distances}. In contrast to the elliptic case that is treated in \cite{GOT20}, a major difficulty in the hypoelliptic case is that one cannot explicitly construct a Cameron-Martin path joining two points in the sense of Definition \ref{def: CM bridge} in any easy way. The analysis of Cameron-Martin norms also becomes more involved. We detail the steps in what follows, starting with some preliminary lemmas.

\subsection{Preliminary results}

As we mentioned above, it is  difficult to explicitly construct a Cameron-Martin path joining $x$ to $y$ in the sense of differential equation in the hypoelliptic case. However, it is possible to find some $u\in\mathfrak{g}^{(l)}$ joining $x$ to $y$ \textit{in the sense of Taylor approximation}, i.e. $y=x+F_l(u,x)$ where the function $F_l(u,x)$ is introduced in Definition~\ref{def: Taylor approximation function}. This is the content of the following lemma proved in \cite[Lemma 3.15]{KS87}. Recall that $l_0$ is the hypoellipticity constant in the assumption \eqref{eq:unif-hypo-assumption}.
 
 \begin{lem}\label{lem: the Psi function}
 For each $l\geq l_{0}$, there exist constants $r,A>0$ depending only
on $l$ and the vector fields, and a $C_{b}^{\infty}$-function 
\[
\Psi_l:\lcl  u\in\mathfrak{g}^{(l)}:\|u\|_{\textsc{HS}}<r\rcl  \times\mathbb{R}^{N}\times\lcl  \eta\in\mathbb{R}^{N}:|\eta|<r\rcl  \rightarrow\mathfrak{g}^{(l)},
\]
such that for all $u,x,\eta$ in the domain of $\Psi_l,$ we have:

\vspace{2mm} 

\noindent (i) $\Psi_l(u,x,0)=u$;\\
\noindent (ii) $\|\Psi_l(u,x,\eta)-u\|_\textsc{HS}\leq A|\eta|;$\\
\noindent (iii) $F_{l}(\Psi_l(u,x,\eta),x)=F_{l}(u,x)+\eta.$
 \end{lem}
 
 \begin{rem}
 The property (ii) is not explicitly stated in  \cite[Lemma 3.15]{KS87}. It is only clear when one develops the construction of $\Psi_l$ carefully. 
 \end{rem}

The intuition behind the function $\Psi_l$ can be described as follows. Let $y\triangleq x+F_l(u,x)$ so that  $u$ joins $x$ to $y$ in the sense of Taylor approximation. Then $v\triangleq\Psi_l(u,x,\eta)$ joins $x$ to $y+\eta$, i.e. $x+F_l(v,x)=y+\eta$. In particular, $\Psi_l(0,x,y-x)$ gives an element in $\mathfrak{g}^{(l)}$ which joins $x$ to $y$ in the sense of Taylor approximation, provided $|y-x|<r$. 

We begin with some preliminary steps toward the proof of Theorem \ref{thm: local comparison}, namely the lower bound on the control distance function $d(\cdot,\cdot)$, and the upper bound for the case $H<1/2$.
\begin{lem}\label{lem: preliminary case}
Assume that the vector fields in equation \eqref{eq: hypoelliptic SDE} satisfy the uniform hypoellipticity assumption \eqref{eq:unif-hypo-assumption} with constant $l_0$. Let $d=d_{H}$ be the control distance introduced in Definition \ref{def: CM bridge}. Then the following bounds hold true.

\vspace{2mm} 

\noindent (i) For all $H\in(1/4,1)$ and $x,y$ such that $|x-y|\leq 1$, we have
$$d(x,y)\geq C_1|x-y|.$$\\
(ii) Whenever $H\in(1/4,1/2)$ we have
$$d(x,y)\leq C_2|x-y|^\frac{1}{l_0}.$$
\end{lem}
\begin{proof}
Claim (i) follows from the exact same argument as in \cite[Theorem 3.3]{GOT20} for the elliptic case.
Claim (ii) stems from the fact that when $H<1/2$, we have
\begin{align}\label{distance upper bound}d(x,y)\leq C_H\, d_\mathrm{BM}(x,y)\end{align}
where $d_{BM}$ stands for the distance for the Brownian motion case. Note that \eqref{distance upper bound} can be easily justified by the fact that
$$d(x,y)\leq \|h\|_{\bar{\mathcal{H}}}\leq C_H\|h\|_{W^{1,2}}$$
for any $h$ joining $x$ and $y$ in the differential equation sense (cf. \cite[Lemma 2.8]{GOT20}). Then, with \eqref{distance upper bound} in hand, our claim (ii) follows from the Brownian hypoelliptic analysis \cite{KS87}.


\end{proof}

From now on, we focus on the case $H>1/2$. It is not surprising that this is the harder case since the Cameron-Martin subspace $\bch$ gets smaller as $H$ increases.  First, we need to make use of the following scaling property of the Cameron-Martin norm. The proof is immediate by using the relation \eqref{eq: inner product in terms of fractional integrals} and a simple change of variables. We denote $\bar{\mathcal{H}}([0,T])$ (respectively, $d_T(x,y)$) as the Cameron-Martin subspace (respectively, the control distance function) associated with fractional Brownian motion over $[0,T]$.

\begin{lem}\label{lem: CM scaling}
Let $0<T_1<T_2$. Given  $h\in\bar{\mathcal{H}}([0,T_1])$, define $\tilde{h}_t\triangleq h_{T_1 t/T_2}$ for $0\leq t\leq T_2$. Then $\tilde{h}\in\bar{\mathcal{H}}[0,T_2]$, and \begin{align}\label{bar-H norm rescaling}
\|\tilde{h}\|_{\bar{\mathcal{H}}([0,T_{2}])}=\lp \frac{T_{1}}{T_{2}}\rp ^{H}\|h\|_{\bar{{\cal H}}([0,T_{1}])}.
\end{align}In particular, we have
\begin{align}\label{distance rescaling}
d_{1}(x,y)=T^{H}d_{T}(x,y),\ \ \ \forall T>0,\ x,y\in\mathbb{R}^{N}.
\end{align}
\end{lem}

We also need the following lemma about the free nilpotent group $G^{(l)}$ which allows us to choose a "regular" path $\gamma$ with $S_l(\gamma)=u$  for all $u\in G^{(l)}$. 
 
\begin{lem}\label{lem: quasi-inverse of signature map}
Let $l\geq1$. For each $M>0$, there exists a constant $C=C_{l,M}>0$, such that for every $u\in G^{(l)}$ with $\|u\|_{\textsc{CC}}\leq M$, one can find a smooth path $\gamma:[0,1]\rightarrow\mathbb{R}^d$ which satisfies: 

\vspace{2mm} 

\noindent (i) $S_l(\gamma)=u$;
\\
(ii) $\dot{\gamma}$ is supported on $[1/3,2/3]$;
\\
(iii) $\|\ddot{\gamma}\|_{\infty;[0,1]}\leq C$.
\end{lem}

\begin{proof}
We first prove the claim for a generic element $u\in\exp(\mathcal{L}_{k})$, seen as an element of $G^{(k)}$. Let $\{a_{1},\ldots,a_{d_{k}}\}$ be a basis of $\mathcal{L}_{k}$
where $d_{k}\triangleq\dim\mathcal{L}_{k}$\,. Given $u\in\exp(\mathcal{L}_{k})$,  we can write $u=\exp(a)$ with
\begin{align}\label{linear combination}a=\lambda_{1}a_{1}+\cdots+\lambda_{d_{k}}a_{d_{k}}\in\mathcal{L}_{k}\end{align}
for some $\lambda_{1},\ldots,\lambda_{d_{k}}\in\mathbb{R}.$ Since
we assume that $\|u\|_{\text{CC}}\leq M,$ according to the ball-box
estimate (cf. Proposition \ref{prop: ball-box estimate}) and the fact that $a\in\mathcal{L}_{k}$, we have 
\begin{align}\label{distance bound for a}
\|a\|_{\mathrm{HS}}=\|u-\mathbf{1}\|_{\mathrm{HS}}\leq C_{1,l,M}.
\end{align}
Moreover, $\mathcal{L}_{k}$ is a finite dimensional vector space, on which all norms are equivalent.  Thus relation \eqref{distance bound for a} yields 
\begin{equation}\label{eq: boundedness of lambda_i}
\max_{1\leq i\leq d_{k}}|\lambda_{i}|\leq C_{2,l,M}.
\end{equation}

Now recall from Remark \ref{connecting by smooth path} that for each $a_{i}$ in \eqref{linear combination} one can  choose
a smooth path $\alpha_i:[0,1]\rightarrow \mathbb{R}^d$ such
that $S_{k}(\alpha_{i})=\mathrm{exp}(a_{i})$ and $\dot{\alpha}_i$ is supported on $[1/3,2/3]$.  Set 
\[R_{k}\triangleq\max\lcl  \|\ddot{\alpha}_{i}\|_{\infty;[0,1]}:1\leq i\leq d_{k}\rcl .\]Note that $R_{k}$ is a constant depending only on $k$.
We construct a smooth path $\gamma:[0,d_{k}]\rightarrow\mathbb{R}^d$ by
\begin{align}\label{definition gamma}
\gamma\triangleq\lp |\lambda_{1}|^{\frac{1}{k}}\alpha_{1}^{\mathrm{sgn}(\lambda_{1})}\rp \sqcup\cdots\sqcup\lp |\lambda_{d_{k}}|^{\frac{1}{k}}\alpha_{d_{k}}^{\mathrm{sgn}(\lambda_{d_{k}})}\rp ,
\end{align}
where $\alpha_{i}^{-1}$ denotes the reverse of $\alpha_{i},$ and $\sqcup$
denotes path concatenation. Then $\dot{\gamma}$ is obviously compactly supported, and we also claim that $S_{k}(\gamma)=u$.  Indeed, it follows from \eqref{definition gamma} that
\begin{align}\label{c0}
S_{k}(\gamma) & =S_{k}\lp |\lambda_{1}|^{\frac{1}{k}}\alpha_{1}^{{\rm sgn}(\lambda_{1})}\rp \otimes\cdots\otimes S_{k}\lp |\lambda_{d_{k}}|^{\frac{1}{k}}\alpha_{d_{k}}^{{\rm sgn}(\lambda_{d_{k}})}\rp 
\notag\\
 & =\delta_{|\lambda_{1}|^{\frac{1}{k}}}\lp S_{k}\lp \alpha_{1}^{{\rm sgn}(\lambda_{1})}\rp \rp \otimes\cdots\otimes\delta_{|\lambda_{d_{k}}|^{\frac{1}{k}}}\lp S_{k}\lp \alpha_{d_{k}}^{{\rm sgn}(\lambda_{d_{k}})}\rp \rp \notag\\
 & =\delta_{|\lambda_{1}|^{\frac{1}{k}}}\lp \exp({\rm sgn}(\lambda_{1})a_{1})\rp \otimes\cdots\otimes\delta_{|\lambda_{d_{k}}|^{\frac{1}{k}}}\lp \exp({\rm sgn}(\lambda_{d_{k}})a_{d_{k}})\rp ,
\end{align}
where we have used the properties of the dilation, recalled in Section \ref{sec:free-nilpotent}, and the relation between signatures and $G^{(l)}$ given in \eqref{eq:def-iterated-intg} -- \eqref{eq:signature-smooth-x}. 
In addition, since each element $\exp(\lambda_{i}a_{i})$ above sits in $\exp(\cl_{k})$, the tensor product in $G^{(k)}$ is reduced to
\begin{equation}\label{c1}
S_{k}(\gamma)
=
\exp(\lambda_{1}a_{1})\otimes\cdots\otimes\exp(\lambda_{d_{k}}a_{d_{k}})
 =\exp(a)=u.
\end{equation}
We have thus found a path $\gamma$ with compactly supported derivative such that $S_{k}(\gamma)=u$.
In addition, from the definition of $R_{k}$ and (\ref{eq: boundedness of lambda_i}), we have
\[
\|\ddot{\gamma}\|_{\infty;[0,d_{k}]}\leq R_{k}\cdot\lp \max_{1\leq i\leq d_{k}}|\lambda_{i}|\rp ^{\frac{1}{k}}\leq C_{3,l,M}.
\]
By suitable rescaling and adding trivial pieces on both ends if necessary,
we may assume that $\gamma$ is defined on $[0,1]$ and $\dot{\gamma}$ is supported on $[1/3,2/3]$. In this way, we have$$\|\ddot{\gamma}\|_{\infty;[0,1]}\leq C_{k}\cdot C_{3,k,M}
\triangleq C_{4,k,M},$$where $C_{k}$ is the constant coming from the rescaling. Therefore, our assertion (i)--(iii)
holds for  $u$ which are elements of $\exp(\mathcal{L}_{k})$.

With the help from the previous special case, we now prove the lemma
by induction on $l$. The case when $l=1$ is obvious, as we can simply
choose $\gamma$ to be a straight line segment. Suppose now that the claim
is true on $G^{(l-1)}$. We let $M>0$ and $u\in G^{(l)}$ with $\|u\|_{\text{CC}}\leq M.$
Define $v\triangleq\pi^{(l-1)}(u)$ where $\pi^{(l-1)}:\ G^{(l)}\rightarrow G^{(l-1)}$
is the canonical projection. We obviously have 
\[
\|v\|_{\text{CC}}\leq\|u\|_{\text{CC}}\leq M,
\]
where the CC-norm of $v$ is taken on the group $G^{(l-1)}$. According
to the induction hypothesis, there exists a constant $C_{l-1,M},$
such that we can find a smooth path $\alpha:[0,1]\rightarrow\mathbb{R}^d$ which satisfies (i)--(iii) in the assertion of Lemma \ref{lem: quasi-inverse of signature map}, for $v=S_{l-1}(\alpha)$ and constant $C_{l-1,M}$. Define 
\begin{align}\label{def: w}w\triangleq\lp S_{l}(\alpha)\rp ^{-1}\otimes u,\end{align}
{where the tensor product is defined on $G^{(l)}$}. Then 
note that owing to the fact that $\|u\|_{\text{CC}}\leq M$, we have
\begin{align*}
\|w\|_{\textsc{CC}} & \leq\|S_{l}(\alpha)\|_{\textsc{CC}}+\|u\|_{\textsc{CC}}\leq\|\alpha\|_{1{\rm -var};[0,1]}+\|u\|_{\textsc{CC}}
  \leq\frac{1}{2}\|\ddot{\alpha}\|_{\infty;[0,1]}+M.
 \end{align*}
 Therefore, thanks to the induction procedure applied to $v=S_{l-1}(\alpha)$, we get
 \begin{align*}
 \|w\|_{\textsc{CC}}\leq \frac{1}{2}C_{l-1,M}+M\triangleq C_{5,l,M}.
\end{align*}

We claim that $w\in\exp(\mathcal{L}_{l})$. This can be proved in the following way. 

\vspace{2mm} 

\noindent(i) Write $u=\exp(l_0+l_h)$, where $l_0\in\mathfrak{g}^{(l-1)}$ and $l_h\in\mathcal{L}_l$. Recall $v\triangleq\pi^{(l-1)}(u)$. We argue  that $v=\exp(l_0)\in G^{(l-1)}$ as follows:  since $l_h\in \mathcal{L}_l$, any product of the form $l^p_h\otimes l^q_0=0$ whenever $p,q>0$. Taking into account the definition \eqref{eq:def-exp-on-T-l} of the exponential function, we get that
\begin{align}\label{u imply v}
u=\exp(l_0+l_h)\ \Longrightarrow\ v=\exp(l_0)\in G^{(l-1)}.
\end{align}

\noindent(ii) Recall that our induction hypothesis asserts that $v=S_{l-1}(\alpha)$, thus according to \eqref{u imply v} we have $S_{l-1}(\alpha)=\exp(l_0)$.  Thanks to the same kind of argument as in (i), we get  $S_l(\alpha)=\exp(l_0+l_h')\in G^{(l)}$ for some $l_h'\in\mathcal{L}_l$. 

\noindent(iii) In order to conclude that $w\in\exp(\mathcal{L}_l)$, we go back to relation \eqref{def: w}, which can now be read as
$$w=\lp \exp(l_0+l'_h)\rp ^{-1}\otimes\exp{(l_0+l_h)}.$$
According to Campbell-Baker-Hausdorff formula and taking into account the fact that
 \[
[l_{0},l_{0}]=[l_{0},l_{h}]=[l_{0},l_{h}']=[l_{h},l_{h}']=0\in\mathfrak{g}^{(l)},
\] we conclude that $w=\exp(l_h-l_h')$ and thus $w\in\exp(\mathcal{L}_l)$.

\vspace{2mm} 

We are now ready to summarize our information and conclude our induction procedure. Namely, for $u\in G^{(l)}$, we can recast relation \eqref{def: w} as
\begin{align}\label{decomposition u}
u=S_l(\alpha)\otimes w,
\end{align}
and we have just proved that $w\in\exp(\mathcal{L}_l)$. Hence relation \eqref{c1} asserts that $w$ can be written as $w=S_l(\beta)$, where $\beta: [0,1]\to \mathbb{R}^d$ satisfying relation (i)-(iii) in Lemma \ref{lem: quasi-inverse of signature map} with $C=C_{6,l,M}$.
Now set $\gamma\triangleq\alpha\sqcup\beta$ and rescale it so that it is defined on $[0,1]$ and its derivative path is supported on $[1/3,2/3]$. Then, recalling our decomposition \eqref{decomposition u}, we have 
\[
S_{l}(\gamma)=S_{l}(\alpha)\otimes S_{l}(\beta)=S_{l}(\alpha)\otimes w=u,
\]
and,  moreover, the following upper bound holds true
\begin{align*}
\|\ddot{\gamma}\|_{\infty;[0,1]} & \leq36\max\lcl  \|\ddot{\alpha}\|_{\infty;[0,1]},\|\ddot{\beta}\|_{\infty;[0,1]}\rcl  \leq C_{7,l,M}. 
\end{align*} Therefore our induction procedure is established, which finishes the proof.

\end{proof}

We conclude this subsection by  stating  a  convention on the group $G^{(l)}$ which will ease notation in our future computations. 

\begin{conv}\label{conv: convention}

Since $\mathfrak{g}^{(l)}$ is a finite dimensional vector space on which differential calculus is easier to manage, we will frequently identify $G^{(l)}$ with $\mathfrak{g}^{(l)}$ through the exponential diffeomorphism without further mention.  
In this way, for instance, $S_l(w)=u$ means $S_l(w)=\exp(u)$ if $u\in\mathfrak{g}^{(l)}$. The same convention will apply to other similar relations when the meaning is clear from context.
For norms on $\mathfrak{g}^{(l)}$, we denote $\|u\|_\mathrm{CC}\triangleq\|\exp(u)\|_{\mathrm{CC}}$. As for the $\rm HS$-norm, note that \[
C_{1,l}\|u\|_{{\rm HS}}\leq\|\exp(u)-{\bf 1}\|_{{\rm HS}}\leq C_{2,l}\|u\|_{{\rm HS}}
\]for all $u\in\mathfrak{g}^{(l)}$ satisfying $\|\exp(u)-{\bf 1}\|_{{\rm HS}}\wedge\|u\|_{{\rm HS}}\leq1$. Therefore, up to a constant depending only on $l$, the notation $\|u\|_{\rm{HS}}$ can either mean the $\rm HS$-norm of $u$ or $\exp(u)-\bf 1$. This will not matter because we are only concerned with local estimates. The same convention applies to the distance functions $\rho_\mathrm{CC}$ and $\rho_\mathrm{HS}$.
\end{conv}

\subsection{Proof of Theorem \ref{thm: local comparison}}
\label{sec:proof-thm-12}

In this section we give the details  to complete the proof of Theorem \ref{thm: local comparison}, namely the local comparison between the distance $d(\cdot,\cdot)$ and the Euclidean distance. 
Thanks to  Lemma \ref{lem: preliminary case}, we only focus on the upper bound for $H>1/2$.

Recall that $\Psi_{l}(u,x,\eta)$ is the function given by Lemma \ref{lem: the Psi function}.
This function allows us to construct elements in $\mathfrak{g}^{(l)}$ joining
two points in the sense of Taylor approximation locally. In what follows,
we take $l=l_{0}$ (where $l_0$ stands for the hypoellipticity constant) and we will omit the subscript $l$ for simplicity (e.g. $F=F_l$ and $\Psi=\Psi_l$) .
We will also identify $G^{(l)}$ with $\mathfrak{g}^{(l)}$ in the
way mentioned in Convention \ref{conv: convention}. We now divide our proof into several steps.

\medskip

\noindent{\it Step 1: Construction of an approximating sequence.}
Let $\delta<r$ be a constant to be chosen later on, where $r$ is
the constant appearing in the domain of $\Psi$ in Lemma \ref{lem: the Psi function}. Consider $x,y\in\mathbb{R}^{N}$ with $|x-y|<\delta.$

We are going to construct three sequences $$\{x_{m}\}\subseteq\mathbb{R}^{N},\ 
\{u_{m}\}\subseteq\mathfrak{g}^{(l_{0})},\ \{h_{m}\}\subseteq C^{\infty}([0,1];\mathbb{R}^{d})$$
inductively. We start with $x_{1}\triangleq x$ and define the rest
of them by the following general procedure in the order $$u_{1}\rightarrow h_{1}\rightarrow x_{2}\rightarrow u_{2}\rightarrow h_{2}\rightarrow x_{3}\rightarrow\cdots.$$
To this aim, suppose we have already defined $x_{m}.$ Set 
\begin{align}\label{def: um}
u_{m}\triangleq\Psi(0,x_{m},y-x_{m}),\quad\text{and}\quad \bar{u}_{m}\triangleq\delta_{\|u_{m}\|_{\text{CC}}^{-1}}u_{m}.
\end{align}
By Lemma \ref{lem: the Psi function}, the first condition in \eqref{def: um} states that $u_m$ is an element of $\mathfrak{g}^{(l_0)}$ such that \begin{align}\label{xm and y}x_m+F(u_m,x_m)=y,\end{align} while the second condition in \eqref{def: um} ensures that  $\|\bar{u}_{m}\|_{\text{CC}}=1$. Once $u_m$ is defined, we construct $h_m$ in the following way: let $\bar{h}_{m}:[0,1]\rightarrow\mathbb{R}^{d}$
be the smooth path given by Lemma~\ref{lem: quasi-inverse of signature map} such that $S_{l_{0}}(\bar{h}_{m})=\bar{u}_{m}$,
$\dot{\bar{h}}_{m}$ is supported on $[1/3,2/3],$ and $\|\ddot{\bar{h}}_{m}\|_{\infty;[0,1]}\leq C_{l_{0}}$.
Define 
\begin{align}\label{def: hm}
h_{m}\triangleq\|u_{m}\|_{\text{CC}}\bar{h}_{m},
\end{align}
so that the truncated signature of $h_m$ is exactly $u_m$ (here recall the Convention \ref{conv: convention}). More specifically, we have:
\[
S_{l_{0}}(h_{m})=S_{l_{0}}(\|u_{m}\|_{\text{CC}}\cdot\bar{h}_{m})=\delta_{\|u_{m}\|_{\text{CC}}}(S_{l_{0}}(\bar{h}_{m}))=\delta_{\|u_{m}\|_{\text{CC}}}(\overline{u}_{m})=u_{m}.
\]
Taking into account  the definition \eqref{eq:def-norm-CC} of the $\textsc{CC}$-norm, it is immediate that
\begin{equation}
\|u_{m}\|_{\text{CC}}\leq\|h_{m}\|_{1\text{-var};[0,1]}\leq\|u_m\|_{\text{CC}}\|\bar{h}_m\|_{1\text{-var};[0,1]}\leq C_{l_{0}}\|u_{m}\|_{\text{CC}}\, ,\label{eq: controlling h_m in terms of u_m CC}
\end{equation}where the last inequality stems from the fact that $\bar{h}_m$ has a bounded second derivative.
Eventually we define 
\begin{align}\label{def: xm}
x_{m+1}\triangleq\Phi_{1}(x_{m};h_{m}),
\end{align}
where recall that $\Phi_{t}(x;h)$ is the solution flow of the ODE
(\ref{eq: ode}) driven by $h$ over $[0,1].$

\medskip
\noindent{\it Step 2: Checking the condition $|y-x_m|< r$.}
Recall that in Lemma \ref{lem: the Psi function} we have to impose $\|u\|_{\text{HS}}<r$ and $|\eta|<r$ in order to apply $\Psi$. In the context of \eqref{def: um} it means that we should make sure that 
\begin{align}\label{xm close to y}|y-x_m|<r, \quad\text{for\ all} \ m.\end{align}
We will now choose $\delta_1$ small enough such that if $|y-x|<\delta_1$, then \eqref{xm close to y} is satisfied.
This will guarantee that $u_{m}$
is well-defined  by Lemma \ref{lem: the Psi function} and we will also be able to write down several useful estimates
for $x_{m}$ and $u_{m}.$ Our first condition on $\delta_1$ is that $\delta_1\leq r$, so that if $|x-y|<\delta_1$, we can define $u_1$ by a direct application of Lemma \ref{lem: the Psi function}. We will now prove by induction that if $\delta_1$ is chosen small enough, then condition \eqref{xm close to y} is satisfied. To this aim, assume that $|x_m-y|<\delta_1$. Then one can apply Lemma \ref{lem: the Psi function} in order to define $u_m, h_m$ and $x_{m+1}$. We also get the following estimate: 
\begin{align}\label{um bound by xm-y}
\|u_{m}\|_{\textsc{HS}}\leq A|x_{m}-y|<A\delta_{1},
\end{align}
where $A$ is the constant appearing in Lemma \ref{lem: the Psi function}. In addition, let us require $\delta_{1}\leq1/A$ so that
$\|u_{m}\|_{\textsc{HS}}\leq1.$ Recalling relations \eqref{xm and y} and \eqref{def: xm} we get
\begin{align*}
|x_{m+1}-y|=|\Phi_1(x_m,h_m)-x_m-F(S_{l_0}(h_m),x_m)|.
\end{align*}
Thus applying successively the Taylor type estimate of \cite[Proposition 10.3]{FV10} and relation~\eqref{eq: controlling h_m in terms of u_m CC} we end up with
$$|x_{m+1}-y| \leq C_{V,l_{0}}\|h_{m}\|_{{\rm 1}-{\rm var};[0,1]}^{1+l_{0}}\leq C_{V,l_{0}}\|u_{m}\|_{\textsc{CC}}^{1+l_{0}}.$$
The quantity $\|u_m\|_{\textsc{CC}}$ above can be bounded thanks to the ball-box estimate of Proposition~\ref{prop: ball-box estimate}, for which we observe that the dominating term in \eqref{eq:bnd-CC-to-HS} is $\rho_{\text{HS}}(g_1,g_2)^{1/l_0}$ since our element $u_m$ is bounded by one in HS-norm. We get 
\begin{align*}
 |x_{m+1}-y|\leq  C_{V, l_0}\|u_m\|_{\textsc{CC}}^{1+l_0}  \leq C_{V,l_{0}}\|u_{m}\|_{\textsc{HS}}^{1+\frac{1}{l_{0}}}  \leq C_{V,l_{0}}A^{1+\frac{1}{l_{0}}}|x_{m}-y|^{1+\frac{1}{l_{0}}}.
\end{align*}
Summarizing our considerations so far, we have obtained the estimate
\begin{equation}
|x_{m+1}-y|\leq C_{1,V,l_{0}}\|u_{m}\|_{\textsc{CC}}^{1+l_{0}}\leq C_{2,V,l_{0}}|x_{m}-y|^{1+\frac{1}{l_{0}}}.\label{eq: x-u recursive estimate}
\end{equation}
On top of the inequalities $\delta_1<r$ and $\delta_1\leq 1/A$ imposed previously, we will also assume that  
$
C_{2,V,l_{0}}\delta_{1}^{1/l_{0}}\leq{1}/{2},
$
which easily yields the relation
\begin{equation}
|x_{m+1}-y|\leq\frac{1}{2}|x_{m}-y|<\frac{1}{2}\delta_{1}<\delta_{1}.\label{eq: x_m recursive estimate}
\end{equation}
For our future computations we will thus set 
$$\delta_{1}\triangleq r\wedge A^{-1}\wedge(2C_{2,V,l_{0}})^{-l_{0}}.
$$
According to our bound \eqref{eq: x_m recursive estimate}, we can guarantee that if $|x-y|<\delta_1$, then $|x_m-y|<\delta_1<r$ for all $m$. In addition, an easy induction procedure performed on inequality \eqref{eq: x_m recursive estimate} leads to the following relation, valid for all $m\geq 1$:

\begin{equation}
|x_{m}-y|\leq2^{-(m-1)}|x-y|.\label{eq: exponential decay of x_m-y}
\end{equation}
Together with the second inequality of (\ref{eq: x-u recursive estimate}),
we obtain that 
\begin{equation}
\|u_{m}\|_{\textsc{CC}}\leq C_{3,V,l_{0}}2^{-\frac{m}{l_{0}}}|x-y|^{\frac{1}{l_{0}}},\ \ \ \forall m\geq1.\label{eq: exponential decay of u_m}
\end{equation}

We will now choose a constant $\delta_2\leq\delta_1$ such that the sequence $\{\|u_m\|_{\text{CC}}; m\geq 1\}$ is decreasing with $m$ when $|x-y|<\delta_2$. This property will be useful for our future considerations. Towards this aim, observe that applying successively \eqref{eq:bnd-CC-to-HS}, \eqref{um bound by xm-y} and \eqref{eq: x-u recursive estimate} we get
\begin{equation}\label{d1}
\|u_{m+1}\|_{\textsc{CC}}\leq C_{l_0}\|u_{m+1}\|_{\text{HS}}^{\frac{1}{l_0}}\leq C_{4,V,l_0}\|u_m\|_{\textsc{CC}}^{1+\frac{1}{l_0}}.
\end{equation}
Hence invoking the second inequality in (\ref{eq: x-u recursive estimate}) we have
\begin{equation}\label{eq: recursive estimate for u_m}
\|u_{m+1}\|_{\text{CC}}\leq  C_{5,V,l_{0}}|x-y|^{\frac{1}{l^2_{0}}}\|u_{m}\|_{\textsc{CC}}.
\end{equation}
Therefore, let us consider a new constant $\delta_{2}>0$ such that 
\begin{equation*}
C_{5,V,l_0}\delta_2^{\frac{1}{l^2_0}}<1.
\end{equation*}
If we choose $|x-y|<\delta$ with  $\delta\triangleq\delta_{1}\wedge\delta_{2},$  equation \eqref{eq: recursive estimate for u_m} can be recast as
\begin{align}\label{monotone um}
\|u_{m+1}\|_{\text{CC}}\leq\|u_m\|_{\text{CC}}.
\end{align}
Note that $\delta=\delta_1\wedge\delta_2$
depends only on $l_0$ and the vector fields, but not on the Hurst parameter $H$. We have thus shown that the application of Lemma \ref{lem: the Psi function} is valid in our context. 

\medskip
\noindent{\it Step 3: Construction of a path joining $x$ and $y$ in the sense of differential equation.} Our next aim is to obtain a path $\tilde{h}$ joining $x$ and $y$ along the flow of equation \eqref{eq: ode}. The first step in this direction is to rescale $h_{m}$ in a suitable way. Namely, set $a_{1}\triangleq0,$
and for $m\geq1$, define recursively the following sequence:
\[
a_{m+1}\triangleq\sum_{k=1}^{m}\|u_m\|_\textsc{CC},\ \ I_{m}\triangleq[a_{m},a_{m+1}],\ \ I\triangleq\overline{\bigcup_{m=1}^{\infty}I_{m}}.
\]
It is clear that $|I_{m}|=\|u_m\|_\textsc{CC},$ and $I$
is a compact interval since the sequence $\{\|u_m\|_\textsc{CC}: m\ge 1\}$ is summable
according to (\ref{eq: exponential decay of u_m}).
We also define a family of function $\{\tilde{h}_m, m\geq1\}$ by
\begin{align}\label{def: tilde hm}
\tilde{h}_{m}(t) & \triangleq h_{m}\lp \frac{t-a_{m}}{a_{m+1}-a_{m}}\rp ,\ \ t\in I_{m},
\end{align}
and the concatenation of the first $\tilde{h}_m$'s is
\begin{align}\label{def: concatenation hm}
\tilde{h}^{(m)}\triangleq\tilde{h}_{1}\sqcup\cdots\sqcup\tilde{h}_{m}:\ [0,a_{m+1}]\rightarrow\mathbb{R}^{d}.
\end{align}
We will now bound the derivative  of $\tilde{h}_{m}.$ Specifically, we first use equation \eqref{def: tilde hm} to get
\[
\sup_{m\geq1}\|\dot{\tilde{h}}^{(m)}\|_{\infty;[0,a_{m+1}]}=\sup_{m\geq1}\|\dot{\tilde{h}}_{m}\|_{\infty;I_{m}}=\sup_{m\geq1}\frac{1}{|I_{m}|}\cdot\|\dot{{h}}_{m}\|_{\infty;[0,1]}.
\]
Then resort to relation \eqref{def: hm}, which yields
\[
\sup_{m\geq1}\|\dot{\tilde{h}}^{(m)}\|_{\infty;[0,a_{m+1}]}=\sup_{m\geq1}\lcl  \frac{\|u_{m}\|_{\text{CC}}}{|I_{m}|}\cdot\|\dot{\bar{h}}_{m}\|_{\infty;[0,1]}\rcl .
\]
Since $\|u_m\|_{\textsc{CC}}=|I_m|$ we end up with
\begin{align}\label{bound derivative tilde hm}
\sup_{m\geq1}\|\dot{\tilde{h}}^{(m)}\|_{\infty;[0,a_{m+1}]}=\sup_{m\geq1}\lcl \|\dot{\bar{h}}_m\|_{\infty; [0,1]}\rcl \leq C_{l_0},
\end{align}
where the last inequality stems from the fact that $\|\ddot{\bar{h}}_m\|_{\infty;[0,1]}\leq C_{l_0}.$ 

We can now proceed to the construction of the announced path joining $x$ and $y$. Namely, set
\begin{align}\label{def: tilde h}
\tilde{h}\triangleq\sqcup_{m=1}^{\infty}\tilde{h}_{m}:I\rightarrow\mathbb{R}^{d}.
\end{align}
Then according to \eqref{bound derivative tilde hm} we have that $\tilde{h}$ is a smooth function from $I$ to $\mr^d$. We  claim that $\Phi_1(x;\tilde{h})=y$, where $\Phi$ has to be understood in the sense of equation \eqref{eq: skeleton ODE}. Indeed, set $$z_{t}=\Phi_t(x; \tilde{h}),\quad t\in I.$$
From the construction of $x_{m}$ in \eqref{def: xm} and the fact that
$\tilde{h}|_{[0,a_{m+1}]}=\tilde{h}^{(m)}$ asserted in \eqref{def: tilde h}, we have
\begin{align}\label{zt at am}
x_{m+1}=x+\sum_{\alpha=1}^{d}\int_{0}^{a_{m+1}}V_{\alpha}(z_{t})d\tilde{h}_{t}^{\alpha}.
\end{align}
Since $x_{m+1}\rightarrow y$ as $m\rightarrow\infty$ which can
be easily seen from (\ref{eq: exponential decay of x_m-y}), one can take
limits in~\eqref{zt at am} to conclude that 
\[
y=x+\sum_{\alpha=1}^{d}\int_{0}^{|I|}V_{\alpha}(z_{t})d\tilde{h}_{t}^{\alpha}.
\]
Therefore, $\tilde{h}$ is a smooth path joining $x$ and $y$ in the sense of differential equations.

\medskip

\noindent{\it Step 4: Strategy for the upper bound.}  According to relation~\eqref{distance rescaling} in Lemma~\ref{lem: CM scaling} on the scaling property, we have
\begin{align}
d(x,y) & =|I|^{H}d_{|I|}(x,y)
\leq|I|^{H}\|\tilde{h}\|_{\bar{{\cal H}}([0,|I|])}\nonumber \\
 & =\lim_{m\rightarrow\infty}\Big( \lp \sum_{k=1}^{m}|I_{k}|\rp ^{H}\|\tilde{h}^{(m)}\|_{\bar{{\cal H}}([0,a_{m+1}])}\Big) ,\label{eq: estimating the control metric}
\end{align}
where the last relation stems from the definition \eqref{def: tilde h} of $\tilde{h}$.

In order to estimate the right hand-side of \eqref{eq: estimating the control metric}, we use the relation \eqref{eq: inner product in terms of fractional integrals} for the Cameron-Martin norm to get
$$\|\tilde{h}^{(m)}\|^2_{\bch([0,a_{m+1}])}=\|K^{-1}\tilde{h}^{(m)}\|^2_{L^2([0,a_{m+1}]; dt)}.$$
Using the definition of $K$ given by \eqref{eq: analytic expression of K}, we are led to
$$\|\tilde{h}^{(m)}\|^2_{\bch([0,a_{m+1}])}=C_H\int_0^{a_{m+1}}\lln t^{H-\frac{1}{2}}D_{0+}^{H-\frac{1}{2}}(s^{\frac{1}{2}-H}\dot{\tilde{h}}^{(m)}(s))(t)\rrn ^{2}dt.$$
By the formula for the fractional derivative (cf. \cite[Equation (2.4)]{GOT20}), we obtain
\begin{align*}
\|\tilde{h}^{(m)}\|_{\bar{{\cal H}}([0,a_{m+1}])}^{2}  
& =C_{H}\cdot\int_{0}^{a_{m+1}}\Big| t^{H-\frac{1}{2}}\lp t^{1-2H}\dot{\tilde{h}}^{(m)}(t)\\
 & \ \ \ +\lp H-\frac{1}{2}\rp \int_{0}^{t}\frac{t^{\frac{1}{2}-H}\dot{\tilde{h}}^{(m)}(t)-s^{\frac{1}{2}-H}\dot{\tilde{h}}^{(m)}(s)}{(t-s)^{H+\frac{1}{2}}}ds\rp \Big| ^{2}dt.
\end{align*}
We now  split the interval $[0,a_{m+1}]$ as $[0, a_{m+1}]=\cup_{k=0}^mI_k$ and use the elementary inequality $(a+b+c)^2\leq 3(a^2+b^2+c^2)$ in order to get
\begin{align}\label{decomposition to Q}\|\tilde{h}^{(m)}\|_{\bar{{\cal H}}([0,a_{m+1}])}^{2}\leq Q_1+Q_2+Q_3,\end{align}
with
\begin{align}
Q_{1} & \triangleq\sum_{k=1}^{m}\int_{I_{k}}\lln t^{H-\frac{1}{2}}\lp t^{1-2H}\dot{\tilde{h}}_{k}(t)\rp \rrn ^{2}dt\triangleq\sum_{k=1}^m Q_{1,k},\label{Q1}\\
Q_{2} & \triangleq\sum_{k=1}^{m}\int_{I_{k}}
\Big| t^{H-\frac{1}{2}}\sum_{l=1}^{k-1}\int_{I_{l}}\frac{t^{\frac{1}{2}-H}\dot{\tilde{h}}_{k}(t)-s^{\frac{1}{2}-H}\dot{\tilde{h}}_{l}(s)}{(t-s)^{H+\frac{1}{2}}}ds
\Big| ^{2}dt\triangleq\sum_{k=1}^mQ_{2,k},\label{Q2}\\
Q_{3} & \triangleq\sum_{k=1}^{m}\int_{I_{k}}
\Big| t^{H-\frac{1}{2}}\int_{a_{k}}^{t}\frac{t^{\frac{1}{2}-H}\dot{\tilde{h}}_{k}(t)-s^{\frac{1}{2}-H}\dot{\tilde{h}}_{k}(s)}{(t-s)^{H+\frac{1}{2}}}ds
\Big| ^{2}dt\triangleq\sum_{k=1}^mQ_{3,k}.\label{Q3}
\end{align}

We are now reduced to bound the above three terms. For the sake of conciseness we will mainly focus on $Q_{2}$, which is the most demanding in terms of singularities. We leave to the patient reader the non-rewarding task of checking details for $Q_{1}$ and $Q_{3}$ or refer to \cite{GOT19} for the complete details.

\medskip
\noindent{\it Step 5: Bound for $Q_2$.} 
In order to estimate $Q_{2}$,  we handle each $Q_{2,k}$ in \eqref{Q2} separately and we resort to the elementary change of variables
\[
u\triangleq\frac{t-a_{k}}{a_{k+1}-a_{k}},\ \quad\text{and}\ \quad v\triangleq\frac{s-a_{l}}{a_{l+1}-a_{l}}.
\]
We also express the terms $\dot{\tilde{h}}_k$ in \eqref{Q2} in terms of $\dot{\bar{h}}_k$. Thanks to some easy algebraic manipulations, we get
\begin{align}\label{express Q2k}
Q_{2,k}=\int_{0}^{1}
\Big| \sum_{l=1}^{k-1}\int_{0}^{1}\frac{\frac{\dot{h}_{k}(u)}{|I_{k}|}-\lp \frac{a_{k}+u|I_{k}|}{a_{l}+v|I_{l}|}\rp ^{H-\frac{1}{2}}\cdot\frac{\dot{h}_{l}(v)}{|I_{l}|}}{(a_{k}+u|I_{k}|-a_{l}-v|I_{l}|)^{H+\frac{1}{2}}}|I_{l}|dv
\Big|^{2}
|I_{k}|du.
\end{align}
In the expression above, notice that for $l\leq k-1$ we have
$$a_k+u|I_k|-a_l-v|I_l|=q_{k,l}(u,v),$$
where
\begin{equation}\label{d11}
q_{k,l}(u,v)=(1-v)|I_l|+|I_{l+1}|+\cdots+|I_{k-1}|+u|I_k|.
\end{equation}
Therefore, invoking the trivial bounds $a_k+u|I_k|\leq \sum_{j_1=1}^k|I_{j_1}|$ and $a_l+v|I_l|\geq\sum_{j_2=1}^{l-1}|I_{j_2}|$, and bounding trivially the differences by sums, we obtain
\begin{align}\label{bound Q2k}
Q_{2,k}\leq C_{H}\int_{0}^{1}
\Big| \sum_{l=1}^{k-1}\int_{0}^{1}\frac{\lln \dot{\bar{h}}_{k}(u)\rrn +\lp \frac{\sum_{j_1=1}^k|I_{j_1}|}{\sum_{j_2=1}^{l-1}|I_{j_2}|}\rp ^{H-\frac{1}{2}}
\cdot\lln \dot{\bar{h}}_{l}(v)\rrn }{|q_{k,l}(u,v)|^{H+\frac{1}{2}}}|I_{l}|dv
\Big|^{2}|I_{k}|du.
\end{align}
In order to obtain a sharp estimate in \eqref{bound Q2k}, we want to take advantage of the fact that $\dot{\bar{h}}_l$ is supported on $[1/3,2/3]$. This allows to avoid the singularities in $u, v$ close to $0$ and $1$. We thus introduce the intervals
\[
J_{1}\triangleq[0,1/3],\ J_{2}\triangleq[1/3,2/3],\ J_{3}\triangleq[2/3,1]
\]
and decompose the expression \eqref{bound Q2k} as follows,
$$Q_{2,k}\leq C_H\sum_{p,q=1}^{3} L_{k,p,q},$$
where the quantity $L_{k,p,q}$ is defined by
\begin{align}\label{def: Lkpq}
L_{k,p,q}\triangleq\int_{J_{p}}
\Big| \sum_{l=1}^{k-1}\int_{J_{q}}
\frac{\lln \dot{\bar{h}}_{k}(u)\rrn +\lp \frac{|I_{1}|+\cdots+|I_{k}|}{|I_{1}|+\cdots+|I_{l-1}|}\rp ^{H-\frac{1}{2}}\cdot\lln \dot{\bar{h}}_{l}(v)\rrn }
{|q_{k,l}(u,v)|^{H+\frac{1}{2}}}|I_{l}|dv
\Big| ^{2}|I_{k}|du,
\end{align}
for all $p,q=1,2,3.$  Notice again that since all the $\dot{\bar{h}}_k$ are  supported
on $[1/3,2/3]$, the only non-vanishing $L_{k,p,q}$'s are those for which  $p=2$ or $q=2$. Let us show how to handle the terms $L_{k,p,q}$ given by \eqref{def: Lkpq}, according to $q=1, 2$ and $q=3$.

Whenever 
 $q=1$ or $q=2$, regardless of the value of $p,$ it is easily seen from \eqref{d11} that we can  bound $q_{k,l}(u,v)$ from below uniformly by $C\sum_{j=l}^{k-1}|I_j|$. Thanks again to the fact that $\dot{\bar{h}}_k$ is uniformly bounded for all $k$, we obtain
\begin{align*}
 & \frac{\lln \dot{\bar{h}}_{k}(u)\rrn +\lp \frac{|I_{1}|+\cdots+|I_{k}|}{|I_{1}|+\cdots+|I_{l-1}|}\rp ^{H-\frac{1}{2}}\cdot\lln \dot{\bar{h}}_{l}(v)\rrn }{q_{k,l}(u,v)^{H+\frac{1}{2}}}\\
 & \leq\frac{C_{H,l_{0}}}{(|I_{l}|+\cdots+|I_{k-1}|)^{H+\frac{1}{2}}}\cdot\lp \frac{|I_{1}|+\cdots+|I_{k}|}{|I_{1}|+\cdots+|I_{l-1}|}\rp ^{H-\frac{1}{2}}.
\end{align*}
Summing the above quantity over $l$ and integrating over $[0,1]$, we end up with
\begin{align*}
L_{k,p,q}\leq C_{H,l_{0}}|I_{k}|\cdot\lp \sum_{l=1}^{k-1}\frac{|I_{l}|}{(|I_{l}|+\cdots+|I_{k-1}|)^{H+\frac{1}{2}}}\cdot\lp \frac{|I_{1}|+\cdots+|I_{k}|}{|I_{1}|+\cdots+|I_{l-1}|}\rp ^{H-\frac{1}{2}}\rp ^{2}.
\end{align*}
By lower bounding the quantity $|I_1|+\cdots+| I_{l-1}|$ above uniformly by $|I_1|$, we get
\begin{align}\label{bound Lkpq}
L_{k,p,q}\leq C_{H,l_{0}}|I_{k}|\cdot\lp \frac{\sum_{j=1}^k|I_{j}|}{|I_{1}|}\rp ^{2H-1}\cdot\lp \sum_{l=1}^{k-1}|I_{l}|^{\frac{1}{2}-H}\rp ^{2}.
\end{align}
Recall that we have shown in \eqref{monotone um} that $m\mapsto\|u_m\|_{\text{CC}}$ is a decreasing sequence. Since $\|u_m\|_{\textsc{CC}}=|I_m|$ we can bound uniformly $\sum_{l=1}^{k-1}|I_l|^{1/2-H}$ by $k|I_1|^{1/2-H}$ and $|I_1|^{-1}\sum_{j=1}^k|I_j|$ by $k$. Plugging this information into \eqref{bound Lkpq} we obtain,
\begin{align}\label{Lkpq q12}L_{k,p,q}\leq C_{H,l_0}k^{2H+1}|I_k|^{2(1-H)},\end{align}
which is our bound for $L_{k,p,q}$ when $q\in\{1,2\}$.


Let us now bound $L_{k,p,q}$ for  $q=3$ and $p=2.$ In this case, going back to the definition~\eqref{def: Lkpq} of $L_{k,p,q}$, we have that $\dot{\bar{h}}_l(v)=0$ for $v\in J_q$. Thus we get
\begin{align}
L_{k,2,3} & =\int_{J_{2}}\lln \sum_{l=1}^{k-1}\int_{J_{3}}\frac{\lln \dot{\bar{h}}_{k}(u)\rrn |I_{l}|dv}{((1-v)|I_{l}|+|I_{l+1}|+\cdots+|I_{k-1}|+u|I_{k}|)^{H+\frac{1}{2}}}\rrn ^{2}|I_{k}|du\nonumber\\
 & \leq C_{H,l_{0}}\lln \sum_{l=1}^{k-1}\int_{J_{3}}\frac{|I_{l}|dv}{((1-v)|I_{l}|+|I_{l+1}|+\cdots+|I_{k}|)^{H+\frac{1}{2}}}\rrn ^{2}\cdot|I_{k}|,\label{Lk23}
\end{align}
where we have used the boundedness of $\dot{\bar{h}}_k$ for the second inequality.
We can now evaluate the above $v$-integral  explicitly, which yields
\begin{align*}
 & \int_{J_{3}}\frac{|I_{l}|dv}{((1-v)|I_{l}|+|I_{l+1}|+\cdots+|I_{k}|)^{H+\frac{1}{2}}}\\
= & \frac{1}{\lp H-\frac{1}{2}\rp }\lp \frac{1}{(|I_{l+1}|+\cdots+|I_{k}|)^{H-\frac{1}{2}}}-\frac{1}{\lp \frac{1}{3}|I_{l}|+|I_{l+1}|+\cdots+|I_{k}|\rp ^{H-\frac{1}{2}}}\rp 
\leq  \frac{C_{H}}{|I_{k}|^{H-\frac{1}{2}}},
\end{align*}
where the second inequality is obtained by lower bounding trivially $|I_{l+1}+\cdots+|I_k|$ by $|I_k|$.
Summing this inequality over $l$ and plugging this information into \eqref{Lk23}, we get
\begin{align}
L_{k,2,3} & \leq C_{H,l_{0}}|I_{k}|\lp \frac{k}{|I_{k}|^{H-\frac{1}{2}}}\rp ^{2}\leq C_{H,l_{0}}k^{2H+1}|I_{k}|^{2(1-H)}.\label{eq: I_23 case}
\end{align}
Summarizing our considerations in this step, we have handled the cases $q=1,2$ and $(q,p)=(3,2)$ in \eqref{Lkpq q12} and \eqref{eq: I_23 case} respectively. 
Therefore, we obtain
\begin{align}\label{bound for Q2}
Q_{2}\leq C_{H,l_0}\sum_{k=1}^{m}k^{2H+1}|I_{k}|^{2(1-H)}.
\end{align}

\medskip
\noindent{\it Step 6: Conclusion.} Let us go back to the decomposition \eqref{decomposition to Q},
plug in our bound \eqref{bound for Q2}  on $Q_{2}$, and recall that similar bounds are available for $Q_{1}$ and $Q_{3}$.
We get\begin{align*}
\|\tilde{h}^{(m)}\|_{\bar{{\cal H}}([0,a_{m+1}])}^{2} & \leq C_{H}(Q_{1}+Q_{2}+Q_{3})\leq C_{H,l_{0}}\sum_{k=1}^{m}k^{2H+1}|I_{k}|^{2(1-H)}.
\end{align*}
In addition, we have  $|I_{k}|=\|u_{k}\|_{\textsc{CC}}$ and relation (\ref{eq: exponential decay of u_m}) asserts that $k\mapsto\|u_k\|_{\textsc{CC}}$ decays exponentially. Thus we get

\begin{eqnarray}
 &  & \lp \sum_{k=1}^m|I_k|\rp ^{2H}\|\tilde{h}^{(m)}\|_{\bar{{\cal H}}([0,a_{m+1}])}^{2}
  \leq C_{H,l_{0}}\lp \sum_{k=1}^{m}|I_{k}|\rp ^{2H}
  \lp \sum_{k=1}^{m}k^{2H+1}|I_{k}|^{2(1-H)}\rp \label{eq: estimating d by I_m} \nonumber\\
 &  & \ \ \ \leq C_{H,V,l_{0}}\lp \sum_{k=1}^{m}2^{-\frac{k}{l_{0}}}\rp ^{2H}\cdot\lp \sum_{k=1}^{m}k^{2H+1}2^{-\frac{2(1-H)}{l_{0}}k}\rp \cdot|x-y|^{\frac{2}{l_{0}}}\\
 &  & \ \ \ \leq C_{H,V,l_0}|x-y|^{\frac{2}{l_0}},\nonumber
\end{eqnarray}
where we have trivially bounded the partial geometric series for the last step.  Moreover the left hand-side of \eqref{eq: estimating d by I_m} converges to a quantity which is lower bounded by $d^2(x,y)$ as $m\to\infty$, thanks to \eqref{eq: estimating the control metric}.  Therefore, letting $m\to\infty$ in \eqref{eq: estimating d by I_m} we have obtained
\begin{equation}\label{d2}
d(x,y)^{2}\leq C_{H,V,l_{0}}|x-y|^{\frac{2}{l_{0}}},
\end{equation}
which concludes our proof of Theorem \ref{thm: local comparison}.

\subsection{Proof of Theorem \ref{thm:equiv-distances}}\label{s:equiv-distances}

The technique developed in the previous proof also provides part of the essential analysis for proving Theorem \ref{thm:equiv-distances}, namely the local Lipschitz equivalence of all control distance functions with different Hurst parameters. Our main idea for proving Theorem \ref{thm:equiv-distances} is to show that the control distance function $d(x,y)$ is locally Lipshictz equivalent to the "distance" function defined by \begin{equation}\label{eq:CanDis}
g(x,y)\triangleq\inf\{\|u\|_{{\rm CC}}:u\in\mathfrak{g}^{(l_0)}\ \text{and }x+F_{l_{0}}(u,x)=y\}.
\end{equation}Note that $g(x,y)$ is canonical in the sense that it does not depend on the Hurst parameter $H$.

The following lemma leads to one direction of the comparison. 
  
\begin{lem}\label{lem: d respect cc}
\label{lem: d < CC}There exist constants $C,\kappa>0$ such that
for any $u\in\mathfrak{g}^{(l_0)}$ with $\|u\|_{\textsc{HS}}<\kappa$,
we have 
\begin{align}\label{d respect cc}
d(x,x+F_{l_0}(u,x))\leq C\|u\|_{\textsc{CC}}.
\end{align}
\end{lem}
\begin{proof}
We only consider the case when $H>1/2$, as the other case follows
from Lemma~2.8 in \cite{GOT20} and the result
for the diffusion case proved in \cite{KS87}. We use the same notation as in the proof of Theorem \ref{thm: local comparison}. In particular, we set up an inductive procedure as in that proof, starting by setting $u_{1}\triangleq u,$ $x_{1}\triangleq x$,  and $y\triangleq x+F_{l_0}(u,x).$
Choose $\kappa_{1}>0$ so that 
\begin{align}\label{cc small imply x close to y}
\|u\|_{\textsc{CC}}<\kappa_{1}\implies|y-x|<\delta,
\end{align}
where $\delta$ is the constant arising in the proof of Theorem \ref{thm: local comparison}. By constructing successively
elements $u_{m}\in\mathfrak{g}^{(l)}$ and intervals $I_{m}$, we obtain exactly as in \eqref{d1} that
\begin{equation}
|I_{m}|=\|u_{m}\|_{\textsc{CC}}\leq C_{V,l_0}\|u_{m-1}\|^{1+\frac{1}{l_0}}=C_{V,l_0}|I_{m-1}|^{1+\frac{1}{l_0}} .\label{eq: recursive estimate of I_m}
\end{equation}
In addition, along the same lines as (\ref{eq: estimating d by I_m}) and \eqref{d2}, we have 
\begin{align}\label{bound d by pieces}
d(x,x+F_{l_0}(u,x))^{2}\leq 
C_{H,l_0}
\lim_{m\to\infty}
\lp \sum_{k=1}^{m}|I_{k}|\rp ^{2H}\lp \sum_{k=1}^{m}k^{2H+1}|I_{k}|^{2(1-H)}\rp .
\end{align}

We now estimate the right hand side of \eqref{bound d by pieces} in a slightly different way from the previous step. Namely let us set $\alpha\triangleq1+1/l_0$. By iterating (\ref{eq: recursive estimate of I_m}),
we obtain that 
\[
|I_{m}|\leq\lp C_{V,l_0}|I_{1}|\rp ^{\alpha^{m-1}},\ \ \ \forall m\geq1.
\]
Therefore, we can bound the two terms on the right hand-side of \eqref{bound d by pieces} as follows:
\begin{align}\label{first term in bound d by pieces}
\sum_{k=1}^{m}|I_{k}|\leq C_{V,l_0}|I_{1}|\cdot\lp \sum_{k=1}^{m}(C_{V,l_0}|I_{1}|)^{\alpha^{k-1}-1}\rp 
\end{align}
and 
\begin{align}\label{second term in bound d by pieces}
\sum_{k=1}^{m}k^{2H+1}|I_{k}|^{2(1-H)}\leq(C_{V,l_0}|I_{1}|)^{2(1-H)}\cdot\lp \sum_{k=1}^{m}k^{2H+1}(C_{V,l_0}|I_{1}|)^{2(1-H)(\alpha^{k-1}-1)}\rp .
\end{align}
To estimate $|I_1|$, recall that $u_1=u$. We further choose $\kappa_{2}>0$ so that 
\begin{align}\label{u cc norm small imply I1 small}
\|u\|_{\textsc{CC}}<\kappa_{2}\implies C_{V,l_0}|I_{1}|=C_{V,l_0}\|u_1\|_{\textsc{CC}}\leq\frac{1}{2}.
\end{align}
By taking $\kappa=\kappa_1\wedge\kappa_2$, we can assume that both \eqref{cc small imply x close to y} and \eqref{u cc norm small imply I1 small} hold.
Also
note that both of the series 
\begin{align}\label{two series}
\sum_{m=1}^{\infty}\lp \frac{1}{2}\rp ^{\alpha^{m-1}-1}\ \quad\text{and}\ \quad \sum_{m=1}^{\infty}m^{2H+1}\lp \frac{1}{2}\rp ^{2(1-H)(\alpha^{m-1}-1)}
\end{align}
are convergent. Therefore, by plugging \eqref{two series} and \eqref{u cc norm small imply I1 small}, and then \eqref{first term in bound d by pieces} and \eqref{second term in bound d by pieces} into \eqref{bound d by pieces}, we have
\begin{equation*}
d(x,x+F_{l_0}(u,x))^{2}  \leq C_{H,V,l_0}|I_{1}|^{2H}\cdot|I_{1}|^{2(1-H)}
 =C_{H,V,l_0}|I_{1}|^{2}
 =C_{H,V,l_0}\|u\|_{\textsc{CC}}^{2} 
\end{equation*}provided that $\|u\|_{\mathrm{CC}}<\kappa$. 
Our result \eqref{d respect cc} thus follows.
\end{proof}

\noindent It is an immediate consequence of Lemma \ref{lem: d respect cc}  that \[
d(x,y)\leq Cg(x,y)
\]provided that $|x-y|$ is small so that the infimum in \eqref{eq:CanDis} can be taken over those $u$'s with $\|u\|_{\mathrm{CC}}<\kappa$. To complete the proof of Theorem \ref{thm:equiv-distances}, it remains to establish the other direction of the above inequality. We now provide the remaining details by using some pathwise estimates from rough path theory.

\begin{proof}[Proof of Theorem \ref{thm:equiv-distances}]

Suppose that $x,y\in\mathbb{R}^{N}$, and let $h\in\Pi_{x,y}$ be
a Cameron-Martin path that joins $x$ to $y$ in the sense of differential
equations, such that $\|h\|_{\bar{{\cal H}}}\leq2d(x,y).$ Set
$u\triangleq\log S_{l_{0}}(h)\in\mathfrak{g}^{(l_{0})}.$ Note that
$u$ does not join $x$ to $y$ in the sense of Taylor approximation.
However, we can use the map $\Psi_{l_{0}}$ to find a $w\in\mathfrak{g}^{(l_{0})}$
which does this, and more precisely $w$ is given by 
\[
w\triangleq\Psi_{l_{0}}(u,x,y-x-F_{l_{0}}(u,x)).
\]
When $|x-y|$ is small, the quantities $d(x,y),$ $\|h\|_{\bar{{\cal H}}},$
$u$ are all small and $w$ is well defined.

According to Proposition \ref{prop: ball-box estimate}, we have 
\[
\big|\|w\|_{{\rm CC}}-\|u\|_{{\rm CC}}\big|\leq\rho_{{\rm CC}}(w,u)\leq C_{1}\|w-u\|_{{\rm HS}}^{1/l_{0}},
\]
and thus 
\[
\|w\|_{{\rm CC}}\leq C_{1}\|w-u\|_{{\rm HS}}^{1/l_{0}}+\|u\|_{{\rm CC}}.
\]
We now estimate the two terms on the right hand side separately.

For the first term, according to Lemma \ref{lem: the Psi function}
(ii), we have 
\[
\|w-u\|_{{\rm HS}}\leq A|y-x-F_{l_{0}}(u,x)|=A|\Phi_{1}(x;h)-x-F_{l_{0}}(u,x)|
\]
with some constant $A$. By the rough path estimates of \cite[Corollary 10.15]{FV10}
and Proposition \ref{prop: variational embedding}, we know that 
\[
|\Phi_{1}(x;h)-x-F_{l_{0}}(u,x)|\leq C_{V,l_{0}}\|h\|_{q\text{-var}}^{\bar{l}_{0}}\leq C_{H,V,l_{0}}\|h\|_{\bar{{\cal H}}}^{\bar{l}_{0}},
\]
where $q\in[1,2)$ and $\bar{l}_{0}$ is an arbitrary number in $(l_{0},l_{0}+1).$
Therefore, we have 
\[
\|w-u\|_{{\rm HS}}\leq AC_{H,V,l_{0}}\|h\|_{\bar{{\cal H}}}^{\bar{l}_{0}}.
\]

For the second term, we claim that $\|u\|_{\textsc{CC}}\leq C_{H,l_0}\|h\|_{\bar{{\cal H}}}.$
Indeed, recall that $u=\log(S_{l_0}(h))$ and set $S_{l_0}(h)=g$. Since $h\in C^{q-\text{var}}$ with $q\in[1,2)$, Lyons' extension theorem  (cf. \cite[Theorem 2.2.1]{Lyons98}) implies that for all $i=1,...,l_0$ we have
\[
\|g_{i}\|_{\textsc{HS}}\leq C_{H,l_0}\|h\|_{q-{\rm var}}^{i}\, ,
\]
where $g_{i}$ is the $i$-th component of $g.$ If we define the
homogeneous norm $\interleave\cdot\interleave$ on $G^{(l_0)}$ by 
\[
\interleave\xi\interleave\triangleq\max_{1\leq i\leq l_0}\|\xi_{i}\|_{\textsc{HS}}^{1/i},\ \ \ \xi\in G^{(l_0)},
\]
we get the following estimate:
\begin{align}\label{bound u cc norm by qvar of h}
\interleave g\interleave\leq C_{H,l_0}\|h\|_{q-\text{var}},\quad\text{and}\quad \|u\|_{\textsc{CC}}\leq C_{H,l_0}\interleave g\interleave,
\end{align}
where the second inequality stems from the equivalence of homogeneous norms in $G^{(l_0)}$ (cf. \cite[Theorem 7.44]{FV10}).
Now combining the two inequalities in \eqref{bound u cc norm by qvar of h} and the variation estimate for Cameron-Martin paths (cf. \cite[Corollary 1]{FV06}), we end up with

\begin{align}\label{bound cc norm of u by H bar norm of h}
\|u\|_{\textsc{CC}}\leq C_{H,l_0}\|h\|_{q-{\rm var}}\leq C_{H,l_0}\|h\|_{\bch}.
\end{align}

Therefore, we arrive at 
\[
\|w\|_{{\rm CC}}\leq C_{H,V,l_{0}}\cdot\big(\|h\|_{\bar{{\cal H}}}^{\bar{l}_{0}/l_{0}}+\|h\|_{\bar{{\cal H}}}\big)\leq C_{H,V,l_{0}}\|h\|_{\bar{{\cal H}}}.
\]
By the fact that $w$ joins $x$ to $y$ in the sense of Taylor approximation as well as 
 the choice of $h$, we conclude that
\[
g(x,y)\leq C_{H,V,l_{0}}d(x,y).
\]
This completes the proof of Theorem \ref{thm:equiv-distances}.

\end{proof}

\section{Local lower estimate for the density of solution}\label{sec: local lower bound}

In this section, we develop the proof of Theorem \ref{thm: local lower estimate} under the uniform hypoellipticity assumption (\ref{eq:unif-hypo-assumption}). Comparing with the elliptic case in \cite{GOT20}, one faces a much more complex situation here. More specifically, the deterministic Malliavin covariance matrix of $X_{t}^{x}$  will not be uniformly non-degenerate (i.e. Lemma 3.6 in \cite{GOT20} is no longer true). Without this key ingredient, the entire elliptic argument will break down and one needs new approaches. Our strategy follows the main philosophy of Kusuoka-Stroock~\cite{KS87} in the diffusion case. However, there are non-trivial challenges in several key steps for the fractional Brownian setting, which require new ideas and methods. 

To increase readability, we first summarize the main strategy of the proof. 
Our analysis starts from the existence of the truncated signature of order $l$ for the fractional Brownian motion (cf. Proposition \ref{prop:fbm-rough-path}). Specifically, with our notation~\eqref{eq:signature-smooth-x} in mind, we write
\begin{align}\label{def: signature of B}
\Gamma_t\triangleq S_l(B)_{0,t}=1+\sum_{i=1}^l\int_{0<t_1<\cdots<t_i<t}dB_{t_1}\otimes \cdots\otimes dB_{t_i}.
\end{align}
In the sequel we will also use the truncated $\mathfrak{g}^{(l)}$-valued {\it log-signature} of $B$, defined by
\begin{align}\label{def: log signature of B}
U^{(l)}_t\triangleq\log S_l(B)_{0,t}.
\end{align}
Notice that $U^{(l)}_t$ features in relation \eqref{eq: formal Taylor expansion}, and more precisely the process
\begin{equation}\label{eq:def-X-l}
X_l(t,x)\triangleq x+F_l(U_{t}^{(l)},x)
\end{equation}
is the Taylor approximation of order $l$ for the solution of the rough equation \eqref{eq: hypoelliptic SDE} in small time (cf. relation \eqref{eq: formal Taylor expansion 2}).

With those preliminary notation in hand, we decompose the strategy towards the proof of Theorem \ref{thm: local lower estimate}  into three major steps.

\vspace{2mm}

\noindent
\textit{Step One}. According to the scaling property of fractional Brownian motion, a precise local lower estimate on the density of $U^{(l)}_t$ can be easily obtained from a general positivity property.

\noindent
\textit{Step Two}. When $l\geq l_0$, the hypoellipticity of the vector fields allows us to obtain a precise local lower estimate on the density of the process $X_l(t,x)$ defined by \eqref{eq:def-X-l} from the estimate on $U^{(l)}_t$ derived in step one.

\noindent
\textit{Step Three}. When $t$ is small, the density of $X_l(t,x)$ is close to the density of the actual solution in a reasonable sense, and the latter inherits the lower estimate obtained in step two. 

\vspace{2mm}

The above philosophy was first proposed by Kusuoka-Stroock \cite{KS87} in the diffusion case. However, in the fractional Brownian setting, there are several difficulties when implementing these steps precisely. Conceptually the main challenge arises from the need of respecting the fractional Brownian scaling and the Cameron-Martin structure in each step in order to obtain sharp estimates. More specifically, for Step 1 we need a new idea to prove the positivity for the density of $U_t^{(l)}$ when the Markov property is not available. For Step 2 we rely on Theorem~\ref{thm:equiv-distances} that we have proven in the Section \ref{s:equiv-distances}, which yields sharp estimates for the density of $X_l(t,x)$. In Step 3,  a new ingredient is needed to prove uniformity for an upper estimate for the density of $X_l(t,x)$ with respect to the degree $l$ of expansion. In the following sections, we develop the above three steps mathematically.

\subsection{Step one: local lower estimate for the signature density of fractional Brownian motion}\label{section: step 1}

We fix $l\geq1$. Recall that the truncated signature process $\Gamma_t$ is defined by \eqref{def: signature of B}. Let $\{{\rm e}_{1},\ldots,{\rm e}_{d}\}$ be
the standard basis of $\mathbb{R}^{d}.$ By viewing this family as vectors
in $\mathfrak{g}^{(l)}\cong T_{\mathbf{1}}G^{(l)}$, we denote the
associated left invariant vector fields on $G^{(l)}$ by $\{\tilde{W}_{1},\ldots,\tilde{W}_{d}\}$. It is standard
(cf. \cite[Remark 7.43]{FV10}) that $\Gamma_{t}$ satisfies
the following intrinsic rough differential equation on $G^{(l)}$:
\begin{equation}\label{eq: SDE for signature}
\begin{cases}
d\Gamma_{t}=\sum_{\alpha=1}^{d}\tilde{W}_{\alpha}(\Gamma_{t})dB_{t}^{\alpha},\\
\Gamma_{0}=\mathbf{1}.
\end{cases}
\end{equation}
Let $U_{t}\triangleq\log\Gamma_{t}\in\mathfrak{g}^{(l)}$ be the truncated
log-signature path, as defined in \eqref{def: log signature of B}. Since $\{\tilde{W}_{1},\ldots,\tilde{W}_{d}\}$ satisfies H\"ormander's
condition by the definition of $\mathfrak{g}^{(l)}$, we know that
$U_{t}$ admits a smooth density with respect to the Lebesgue measure
$du$ on $\mathfrak{g}^{(l)}$. Denote this density by $\rho_{t}(u).$ 

A key ingredient in this step is to show that  the density function $\rho_t$ is everywhere strictly positive.  In the Brownian case, this was proved in \cite{KS87} by using support theorem and the Markov property. In the fractional Brownian setting, the argument breaks down although general support theorems for Gaussian rough paths are still available. It turns out that there is a simple neat proof based on Sard's theorem and a general positivity criterion established Baudoin-Nualart-Ouyang-Tindel \cite{BNOT16}. We mention that Baudoin-Feng-Ouyang \cite{BFO19} also contains a proof of this property using a different approach.

We first recall the classical Sard's theorem, and we refer the reader to \cite{Milnor97} for a beautiful presentation. Let $f:M\rightarrow N$ be a smooth map between two finite dimensional differentiable manifolds $M$ and $N$. A point $x\in M$ is said to be a \textit{critical point} of $f$ if the differential $df_x:T_x M\rightarrow T_{f(x)}N$ is not surjective. A \textit{critical value} of $f$ in $N$ is the image of a critical point in $M$. Also recall that a subset $E\subseteq N$ is a \textit{Lebesgue null set} if its intersection with any coordinate chart has zero Lebesgue measure in the corresponding coordinate space. 

\begin{thm}[Sard's theorem]\label{thm: Sard's theorem}
Let $f:M\rightarrow N$ be a smooth map between two finite dimensional differentiable manifolds. Then the set of critical values of $f$ is a Lebesgue null set in $N$.
\end{thm}
  
We now prove the positivity result announced above, which will be important for our future considerations.

\begin{lem}\label{lem: positivity of rho}
For each $t>0$, the density $\rho_t$ of the truncated signature path $U_t$ is everywhere strictly positive.
\end{lem}
\begin{proof}
We only consider the case when $t=1$. The general case follows from the scaling property (\ref{eq: formula for rho_t}) below. 
Our strategy relies on the fact that $\Gamma_t=\exp(U_t)$ solves equation \eqref{eq: SDE for signature}. In addition, recall our Convention \ref{conv: convention} about the identification of $\mathfrak{g}^{(l)}$ and $G^{(l)}$. Therefore we can get the desired positivity by applying \cite[Theorem 1.4]{BNOT16}. To this aim, recall that the standing assumptions in \cite[Theorem 1.4]{BNOT16} are the following:

\vspace{2mm} 

\noindent (i) The Malliavin covariance matrix of $U_t$ is invertible with inverse in $L^p(\Omega)$ for all $p>1$;\\
(ii) for any $u\in\mathfrak{g}^{(l)}$, there exists $h\in\bch$ such that $\log S_l(h)=u$ and
\begin{align}\label{nondegenerate map}
(d\log S_l)_h: \bch\to \mathfrak{g}^{(l)}\quad \text{is\ surjective,}
\end{align}
where $S_l(h)\triangleq S_l(h)_{0,1}$ is the truncated map.

\vspace{2mm} 

\noindent Notice that item (i) is proved in \cite[Theorem 3.3]{BFO19}. We will thus focus on condition (ii) in the remainder of the proof.

In order to prove relation \eqref{nondegenerate map} in item (ii) above, let us introduce some additional notation. First we shall write $G\triangleq G^{(l)}$ for the sake of simplicity. Then for all $n\geq 1$ we introduce a linear map
$
H_{n}:(\mathbb{R}^{d})^{n}\rightarrow\bar{{\cal H}}
$ in the following way. Given $y=(y_1,\ldots,y_n)$, the function $H_n(y)$ is defined to be the piecewise linear path obtained by concatenating the vectors $y_1,\ldots,y_n$ successively. We also define a set $\bch_0$ of piecewise linear paths by
\[
\bar{{\cal H}}_{0}\triangleq\bigcup_{n=1}^{\infty}H_{n}\lp (\mathbb{R}^{d})^{n}\rp \subseteq\bar{\cal H}.
\]Note that $\bar{\mathcal{H}}_0$ is closed under concatenation, and $S_l(\bar{\cal H}_0)=G$ by the Chow-Rashevskii theorem (cf. Remark \ref{connecting by smooth path}). Now we claim that:

\vspace{2mm}

\noindent (\textbf{P}) For any $g\in G$, there exists $h\in\bar{\cal H}_0$ such that $S_l(h)=g$ and the differential $(dS_l)_h|_{\bar{\cal H}_0}:\bar{\cal H}_0\rightarrow T_g G$ is surjective. 

\vspace{2mm}

The property  (\textbf{P}) is clearly stronger than the original desired claim \eqref{nondegenerate map}. To prove (\textbf{P}), let $\cal P$ be the set of elements in $G$ which satisfy (\textbf{P}). We first show that $\cal P$ is either $\emptyset$ or $G$.  The main idea behind our strategy is that if there exists $g_0\in\cal{P}$, such that $(dS_l)_{h_0}$ is a submersion for some $h_0\in\bch_0$ satisfying $S_l(h_0)=g_0$, then one can obtain every point $g\in G$ by a left translation $L_a$, since $dL_a$ is an isomorphism. To be more precise, suppose that $g_{0}\in G$ is an element satisfying (\textbf{P}). By definition, there exists a path $h_{0}\in\bar{{\cal H}}_{0}$ such
that $S_{l}(h_{0})=g_{0}$ and $(dS_{l})_{h_{0}}|_{\bar{{\cal H}}_{0}}$
is surjective. Now pick a generic element $a\in G$ and choose a path
$\alpha\in\bar{{\cal H}}_{0}$ so that $S_{l}(\alpha)=a$. Then $S_{l}(\alpha\sqcup h_{0})=a\otimes g_{0}.$
We want to show that $(dS_{l})_{\alpha\sqcup h_{0}}:\bar{{\cal H}}_{0}\rightarrow T_{a\otimes g_{0}}G$
is surjective. For this purpose, let $\xi\in T_{a\otimes g_{0}}G$
and set 
\[
\xi_{0}\triangleq dL_{a^{-1}}(\xi)\in T_{g_{0}}G.
\]
By the surjectivity of $(dS_{l})_{h_{0}}|_{\bar{{\cal H}}_{0}},$
there exists $\gamma\in\bar{{\cal H}}_{0}$ such that $(dS_{l})_{h_{0}}(\gamma)=\xi_{0}.$
It follows that, for $\varepsilon>0$ we have 
\[
S_{l}(\alpha\sqcup(h_{0}+\varepsilon\cdot\gamma))=a\otimes S_{l}(h_{0}+\varepsilon\cdot\gamma).
\]
By differentiation with respect to $\varepsilon$ at $\varepsilon=0$, we obtain that
\[
(dS_{l})_{\alpha\sqcup h_{0}}(0\sqcup\gamma)=(dL_{a})_{S_{l}(h_{0})}\circ(dS_{l})_{h_{0}}(\gamma)=(dL_{a})_{g_{0}}(\xi_{0})=\xi.
\]
Therefore, $(dS_{l})_{\alpha\sqcup h_{0}}|_{\bar{{\cal H}}_{0}}$
is surjective. Since $a$ is arbitrary, we conclude that if ${\cal P}$
is non-empty, then ${\cal P}=G$.

To complete the proof, it remains to show that $\cal P\neq\emptyset$. This will be a simple consequence of Sard's theorem. Indeed, for each $n\geq1$,  define
\begin{align}\label{def: fn}
f_{n}\triangleq S_{l}\circ H_{n}:(\mathbb{R}^{d})^{n}\rightarrow G,
\end{align}
where we recall that $H_n(y)$ is the piecewise linear path obtained by concatenating $y_1,..., y_n$.
The map $f_n$ is simply given by \[
f_{n}(y_{1},\ldots,y_{n})=\exp(y_{1})\otimes\cdots\otimes\exp({y_{n}}),
\]
where we recall that the exponential maps is defined by \eqref{eq:def-exp-on-T-l}. It is readily checked that $f_n$ is a smooth map.
According to Sard's theorem (cf. Theorem \ref{thm: Sard's theorem}), the set of critical values of $f_n$, denoted as $E_n$, is a Lebesgue null set in $G$. It follows that $E\triangleq\cup_{n=1}^{\infty}E_{n}$ is also a Lebesgue null set in $G$. We have thus obtained that, \[
G\backslash E=\Big( \bigcup_{n=1}^{\infty}f_{n}((\mathbb{R}^{d})^{n})\Big) \backslash E\neq\emptyset,
\]
where the first equality is due to the fact that $S_l(\bch_0)=G$ by the Chow-Rashevskii theorem. 
Pick any element $g\in G\backslash E$. Then for some $n\geq1$, we have $g\in f_n((\mathbb{R}^d)^n)\backslash E_n$. In particular, there exists $y\in (\mathbb{R}^d)^n$ such that $f_n(y)=g$ and $(df_n)_y$ is surjective.  We claim that $g\in\cal P$ with $h\triangleq H_n(y)\in\bar{\cal H}_0$ being the associated path.  Indeed, it is apparent that $S_l(h)=g$. In addition, let $\xi\in T_g G$ and $w\in (\mathbb{R}^d)^n$ be such that $(df_n)_y(w)=\xi$. The existence of $w$ follows from the surjectivity of $(df_n)_y$. Since $H_n$ is linear, we obtain that \begin{align*}
(dS_{l})_{h}(H_{n}(w)) & =\frac{d}{d\varepsilon}\rrn _{\varepsilon=0}S_{l}(H_{n}(y)+\varepsilon\cdot H_{n}(w))
  =\frac{d}{d\varepsilon}\rrn _{\varepsilon=0}S_{l}(H_{n}(y+\varepsilon\cdot w))\\
 & =\frac{d}{d\varepsilon}\rrn _{\varepsilon=0}f_{n}(y+\varepsilon\cdot w)
  =(df_{n})_{y}(w)
  =\xi.
\end{align*}Therefore, the pair $(h,g)$ satisfies property (\textbf{P}) and thus $\cal P$ is non-empty.
\end{proof}

\begin{rem}
{Theorem 1.4 in \cite{BNOT16} was stated for SDEs in which the vector fields are of class $C_b^\infty$. However, the uniform boundedness assumption was not relevant since the argument was essentially local. To elaborate this, first note that Condition (b) in \cite[Theorem 3.1]{BNOT16}, which is the key ingredient for the proof of \cite[Theorem 1.4]{BNOT16}, was stated under convergence in probability. In addition, Condition (b) and the result of Theorem 3.3 in the same paper were both local since the path $h$ is fixed. Therefore, one can  localize the relevant probabilities in a large compact subset and apply local continuity theorems in rough path theory even though the underlying vector fields are not bounded. In particular, the result applies to our truncated signature process $U_t$.}
\end{rem}
\begin{rem}
When $H>1/2$, it is not clear whether $\bar{\cal H}$ contains the space of piecewise linear paths. It is though true from $\bar{\cal H}=I_{0+}^{H+1/2}(L^2([0,1]))$ that it contains all smooth paths. To fix this issue, one can reparametrize the piecewise linear path $y_1\sqcup\cdots\sqcup y_n$ in an obvious way such that the resulting path is smooth but the trajectory remains unchanged. This does not change the truncated signature as it is invariant under reparametrization. 
\end{rem}

Essentially the same amount of effort allows us to adapt the above argument to establish the general positivity result for hypoelliptic SDEs as stated in Theorem \ref{prop: positivity-intro}, which is of independent interest. This complements the result of \cite[Theorem 1.4]{BNOT16} by affirming that Hypothesis 1.2 in that theorem is always verified under hypoellipticity.  

\begin{proof}[Proof of Theorem \ref{prop: positivity-intro}]
Without loss of generality we only consider $t=1$.  Continuing to denote by $\Phi_t(x;h)$ the skeleton of equation \eqref{eq: hypoelliptic SDE}, defined by \eqref{eq: skeleton ODE}, 
let $F:\bar{{\cal H}}\rightarrow\mathbb{R}^{N}$ be the end point
map defined by $F(h)\triangleq\Phi_{1}(x;h)$.  As in the proof of Lemma \ref{lem: positivity of rho}, the key ingredient that needs to be established is the following property:
for any $y\in\mr^N$ there exists $h\in\bch$ such that 
\begin{align}\label{surjectivity general case}
F(h)=y,\quad\text{and}\ \ \ (dF)_h: \bch\to\mr^N \ \ \textnormal{is surjective.}
\end{align}
Along the same lines as in the proof of Lemma \ref{lem: positivity of rho}, we define ${\cal P}$ to be the set of points $y\in\mr^N$ satisfying \eqref{surjectivity general case} for some $h\in\bar{\cal H}$. We first show that ${\cal P}$ is non-empty which then implies $\mathcal{P}=\mathbb{R}^N$ again by a translation argument.

To show that $\cal P$ is non-empty, we first define $H_{n}:(\mathbb{R}^{d})^{n}\rightarrow\bar{{\cal H}}$
and $\bar{{\cal H}}_{0}\subseteq\bar{{\cal H}}$ in the same way as
in the proof of Lemma \ref{lem: positivity of rho}. Also define a map $F_n$ by
\[
F_{n}\triangleq F\circ H_{n}:(\mathbb{R}^{d})^{n}\rightarrow\mathbb{R}^{N}.
\]
According to Sard's theorem, the set of critical values of $F_{n},$
 again denoted as $E_{n}$, is a Lebesgue null set in $\mathbb{R}^{N},$
and so is $E\triangleq\cup_{n}E_{n}$. 

Next consider a given $q\in\mr^N$. Thanks to the hypoellipticity assumption \eqref{eq:unif-hypo-assumption}, we can equip a neighborhood $U_q$ of $q$ with a sub-Riemannian metric, {by requiring that a certain subset of $\{V_1,...,V_d\}$ is an orthonormal frame near $q$}. Then according to the Chow-Rashevskii theorem (cf. \cite{Montgomery02},
Theorem 2.1.2), every point in $U_{q}$ is reachable from $q$ by
a horizontal path. And if one examines the proof of the theorem in
Section 2.4 of \cite{Montgomery02} carefully, this horizontal path
is controlled by a piecewise linear path in $\mathbb{R}^{d}$, i.e.
$U_{q}\subseteq\cup_{n}\Phi_{1}(q;H_{n}((\mathbb{R}^{d})^{n}))$.
Now for given $y\in\mathbb{R}^{N},$ choose an arbitrary continuous
path $\gamma$ joining $x$ to $y$. By compactness, we can cover
the image of $\gamma$ by finitely many open sets of the form $U_{q_{i}}$
such that $U_{q_{i}}\cap U_{q_{i+1}}\neq\emptyset$ for all $i$ where
$q_{i}\in{\rm Im}(\gamma).$ It follows that $y$ can be reached from
$x$ by a horizontal path controlled by a piecewise linear path in
$\mathbb{R}^{d}.$ In other words, we have $y\in F_{n}((\mathbb{R}^{d})^{n})$
for some $n.$ This establishes the property that $\mathbb{R}^{N}=\cup_{n}F_{n}((\mathbb{R}^{d})^{n})$.

Now the same argument as in the proof of Lemma \ref{lem: positivity of rho} allows us to conclude that 
\[
\mathbb{R}^{N}\backslash E=\bigcup_{n=1}^{\infty}F_{n}((\mathbb{R}^{d})^{n})\backslash E\subseteq\mathcal{P},
\]showing that $\cal P$ is non-empty since $E$ is a Lebesgue null set.

Finally, we show that $\mathcal{P}=\mathbb{R}^N$. To this aim, first note that, for any $h_0,\gamma,\alpha\in\bar{\cal H}$ and $\varepsilon>0$, we have \[
\Phi_{1}(x;(h_{0}+\varepsilon\cdot\gamma)\sqcup\alpha)=\Phi_{1}\lp \Phi_{1}(x;h_{0}+\varepsilon\cdot\gamma);\alpha\rp ,
\]where paths are always assumed to be parametrized on $[0,1]$. Therefore, by differentiating with respect to $\varepsilon $ at $\varepsilon=0$, we obtain that \[
(dF)_{h_{0}\sqcup\alpha}(\gamma\sqcup0)=J_{1}(F(h_{0});\alpha)\circ(dF)_{h_{0}}(\gamma),
\]where recall that $J_t(\cdot;\cdot)$ is the Jacobian of the flow $\Phi_t$. This shows that 
\begin{equation}\label{eq:Surj}
(dF)_{h_{0}\sqcup\alpha}=J_{1}(F(h_{0});\alpha)\circ(dF)_{h_{0}}.
\end{equation}

Now pick any fixed $y_0\in\cal P$ with an associated $h_0\in\bar{\cal H}$ satisfying (\ref{surjectivity general case}). For any $\eta\in\mathbb{R}^N$, choose $\alpha\in\bar{\cal H}$ such that $F(\alpha)=\eta.$ Then $F(h_0\sqcup\alpha)=y+\eta$ and the surjectivity of $(dF)_{h_0\sqcup\alpha}$ follows from (\ref{eq:Surj}), the surjectivity of $(dF)_{h_0}$ and the invertibility of the Jacobian. In particular, $y+\eta\in\cal P$. Since $\eta$ is arbitrary, we conclude that $\mathcal{P}=\mathbb{R}^N$.

\end{proof}

\begin{rem}
A general support theorem for hypoelliptic SDEs allows one to show that the support of the density $p_t(x,y)$ is dense. In the diffusion case, together with the semigroup property
 \[
p(s+t,x,y)=\int_{\mathbb{R}^{N}}p(s,x,z)p(t,z,y)dz
\]one immediately sees that $p(t,x,y)$ is everywhere strictly positive. This argument clearly breaks down in the fractional Brownian setting.
\end{rem}

\begin{rem}
In Theorem \ref{prop: positivity-intro}, although the time horizon is taken to be $[0,1]$, the proof clearly applies to arbitrary time horizon $[0,T]$ without difficulty. As for the study of the control distance function,  relation (\ref{distance rescaling}) also allows us to restrict on the time horizon $[0,1]$.
\end{rem}

Finally, we present the main result in this part which gives a precise local lower estimate for the density $\rho_t(u)$. This will be a consequence of Lemma \ref{lem: positivity of rho} together with the scaling property and left invariance for the process $\Gamma_t$.


\begin{prop}\label{prop: step one}
For each $M>0,$ define $\beta_{M}\triangleq\inf\lcl  \rho_{1}(u):\|u\|_{\textsc{CC}}\leq M\rcl $. Then $\beta_M$ is strictly positive and for all $(u,t)\in\mathfrak{g}^{(l)}\times(0,1]$ with $\|u\|_{\textsc{CC}}\leq Mt^{H},$ 
 we have 
\begin{align}\label{signature density lower bound}
\rho_{t}(u)\geq\beta_{M}t^{-H\nu},
\end{align}
where the constant $\nu$ is given by $\nu\triangleq\sum_{k=1}^{l}k\dim{\cal L}_{k},$ and ${\cal L}_k$ is the space of homogeneous Lie polynomials of degree $k$.
\end{prop}
\begin{proof}
First observe that the strict positivity of $\beta_M$ is a direct consequence of Lemma~\ref{lem: positivity of rho} and the compactness of the set $\{u\in\frak{g}^{(l)}; \|u\|_{\textsc{CC}}\leq M\}$. 
Next, using the scaling property of  fractional Brownian motion and the left invariance of the vector fields 
$\tilde{W}_\alpha$ defining the equation \eqref{eq: SDE for signature} for $\Gamma_t$, it is not hard to see that 
$$(\delta_{\lambda}\Gamma_{t})_{0\leq t\leq1}\stackrel{{\rm law}}{=}(\Gamma_{\lambda^{1/H}t})_{0\leq t\leq 1},$$ where recall that $\delta_\lambda$ is the dilation operator on $T^{(l)}$. As a result, if we define $Q_{t}$ to be the law of $U_{t}$ on $\mathfrak{g}^{(l)}$,
then $Q_{s}\circ\delta_{\lambda}^{-1}=Q_{\lambda^{1/H}s}$ for all $s\in[0,1]$.
In particular, by setting $s=1$ and $\lambda=t^{H},$ we obtain that
\begin{equation}\label{eq: scaling property of U_t}
Q_{t}=Q_{1}\circ\delta_{t^{H}}^{-1}.
\end{equation}
It follows that, for any $f\in C_{b}^{\infty}(\mathfrak{g}^{(l)}),$
we have 
\begin{align*}
\int_{\mathfrak{g}^{(l)}}f(u)\rho_{t}(u)du & =\int_{\mathfrak{g}^{(l)}}f(u)Q_{t}(du)
  =\int_{\mathfrak{g}^{(l)}}f(\delta_{t^{H}}u)Q_{1}(du)\\
 & =\int_{\mathfrak{g}^{(l)}}f(\delta_{t^{H}}u)\rho_{1}(u)du
  =\int_{\mathfrak{g}^{(l)}}t^{-H\nu}f(u)\rho_{1}(\delta_{t^{-H}}u)du,
\end{align*}
where the last equality follows from the change of variables $u\leftrightarrow\delta_{t^{-H}}u$
and the fact that $du\circ\delta_{t^{H}}^{-1}=t^{-H\nu}du$ (cf. relation \eqref{eq:dilation-lebesgue-on-cal-G}). Therefore,
we conclude that 
\begin{equation}\label{eq: formula for rho_t}
\rho_{t}(u)=t^{-H\nu}\rho_{1}(\delta^{-1}_{t^{H}}u),\ \ \text{for\ all}\ \ (u,t)\in\mathfrak{g}^{(l)}\times(0,1],
\end{equation}
from which our result \eqref{signature density lower bound} follows.
\end{proof}

\subsection{Step two: local lower estimate for the density of the Taylor approximation process}\label{sec: step two}

Recall from \eqref{eq:def-X-l} that $X_l(t,x)=x+F_l(U_t^{(l)},x)$ is the Taylor approximation process of order $l$ for the actual solution of the SDE (\ref{eq: hypoelliptic SDE}). Due to hypoellipticity, it is natural to expect that $F_l$ is "non-degenerate" in a suitable sense provided  $l\geq l_0$. In addition, a precise local lower estimate for the density of $X_l(t,x)$ should naturally follow from Proposition \ref{prop: step one} in Step One,  combined with such "non-degeneracy" property of $F_l$. Here the main subtlety lies in finding a way of respecting the fractional Brownian scaling and Cameron-Martin structure so that the estimate we obtain on $X_l(t,x)$ is sharp. However, this will be a direct consequence of Theorem \ref{thm:equiv-distances} together with a result for the diffusion case in \cite{KS87}. In this part, we always fix $l\geq l_0$.

We first give the precise meaning of the non-degeneracy of $F_l$. Let $JF_l(u,x):\mathfrak{g}^{(l)}\rightarrow\mathbb{R}^N$ be the Jacobian of $F_l$ with respect to $u$. Since $\mathfrak{g}^{(l)}$ has a canonical Hilbert structure induced from $T^{(l)}(\mathbb R^d)$, we can also consider the adjoint map $JF_l(u,x)^*:\mathbb{R}^N\rightarrow\mathfrak g^{(l)}$.  The non-degeneracy of $JF_l$ is summarized in the following lemma, which was proved in \cite[Lemma 3.13]{KS87}.

\begin{lem}\label{lem: nondegeneracy of F}Let $F_l$ be the approximation map given in Definition \ref{def: Taylor approximation function} and let $JF_l(u,x): \frak{g}^{(l)}\to \mr^N$ be its Jacobian. Then there exists a constant $c>0$ depending only on $l_0$ and the vector fields, such that \[
JF_l(0,x)\cdot JF_l(0,x)^{*}\geq c\cdot\mathrm{Id}_{\mathbb{R}^{N}}
\]for all $l\geq l_0$ and $x\in \mathbb R^N$.
\end{lem}

\noindent An immediate corollary of Lemma \ref{lem: nondegeneracy of F} is the following.

\begin{cor}\label{cor: local submersion}
Given $l\geq l_0,$ there exists $r>0$ depending on $l$ and the vector fields, such that $\det(JF_l(u,x)\cdot JF_l(u,x)^*)$ is uniformly positive on $\{u\in\mathfrak{g}^{(l)}:\|u\|_\mathrm{HS}< r\}\times\mathbb{R}^N$. In particular, the map
\[
\{u\in\mathfrak{g}^{(l)}:\|u\|_{\textsc{HS}}<r\}\rightarrow\mathbb{R}^{N},\qquad u\mapsto x+F_{l}(u,x),
\]is a submersion in the sense of differential geometry. 
\end{cor}

\begin{rem}\label{rem: consistent choice of r}
Note that the map $F_l$ and the constant $r$ in Corollary \ref{cor: local submersion} depend on $l$. For technical reasons, we will assume that $r$ is chosen (still depending on $l$) so that for all $l_0\leq l'\leq l$, the map $JF_{l'}(\pi^{(l')}(u),x)$ has full rank whenever $(u,x)\in\mathfrak{g}^{(l)}\times\mathbb{R}^N$ with $\|u\|_{\mathrm{HS}}<r$, where $\pi^{(l')}:\mathfrak{g}^{(l)}\rightarrow\mathfrak{g}^{(l')}$ is the canonical projection. This property will be used in the proof of Lemma \ref{lem: upper estimate for approximating density} in Step Three below.
\end{rem}

Now let $r$ be the constant given in Remark \ref{rem: consistent choice of r} . It is standard from differential geometry that for each $x\in\mathbb{R}^N$ and $y\in\{x+F_l(u,x):\|u\|_\mathrm{HS}<r\}$, the "bridge space"
\begin{align}\label{Mxy}M_{x,y}\triangleq \{u\in\mathfrak{g}^{(l)}:\|u\|_\mathrm{HS}<r\ \mathrm{and}\ x+F_l(u,x)=y\}\end{align} is a submanifold of $\{u\in\mathfrak{g}^{(l)}:\|u\|_{\mathrm{HS}}<r\}$ with dimension $\dim \mathfrak{g}^{(l)}-N$. In addition, since both of $\mathfrak{g}^{(l)}$ and $\mathbb{R}^N$ are oriented Riemannian manifolds, we know from differential topology that $M_{x,y}$ carries a natural orientation and hence a volume form which we denote as $m_{x,y}$. The following result is the standard disintegration formula in Riemannian geometry (cf. \cite[equation (0.3)]{Bis84}).

\begin{prop}\label{th: disintegration}For any $\varphi\in C_c^\infty(\{u\in\mathfrak{g}^{(l)}:\|u\|_\mathrm{HS}<r\})$, we have 
\begin{equation}\label{eq: disintegration}
\int_{\mathfrak{g}^{(l)}}\varphi(u)du
=
\int_{\mathbb{R}^{N}}dy\int_{M_{x,y}}K(v,x)\varphi(v) \, m_{x,y}(dv),
\end{equation}where the kernel $K$ is given by
\begin{align}\label{kernel in disintegration}
K(v,x)\triangleq\lp \det(JF_{l}(v,x)\cdot JF_{l}(v,x)^{*})\rp ^{-\frac{1}{2}},
\end{align}
and we define $m_{x,y}\triangleq0$ if $M_{x,y}=\emptyset.$
\end{prop}

\noindent The disintegration formula (\ref{eq: disintegration}) immediately leads to a formula for the (localized) density of the Taylor approximation process $X_l(t,x)$.

\begin{prop}\label{prop: representation of density} Let $\eta\in C_c^\infty(\{u\in\mathfrak{g}^{(l)}:\|u\|_\mathrm{HS}<r\})$ be a bump function so that $0\leq\eta\leq1$ and $\eta=1$ when $\|u\|_\mathrm{HS}<r/2$, where $r$ is the constant featuring in Proposition~\ref{th: disintegration}. Define $\mathbb{P}^\eta_l(t,x,\cdot)$ to be the measure \[
\mathbb{P}_{l}^{\eta}(t,x, A)\triangleq\mathbb{E}\lc \eta(U_{t}){\bf 1}_{\{X_{l}(t,x)\in A\}}\rc  ,\ \ \ A\in{\cal B}(\mathbb{R}^{N}),
\]where $U_t=\log \Gamma_t$, $\Gamma_t$ is defined by \eqref{eq: SDE for signature} and $X_{l}(t,x)=x+F_{l}(U_{t},x)$ is the approximation given by \eqref{eq:def-rough-eq}.
Then the measure $\mathbb{P}_{l}^{\eta}(t,x,\cdot)$ is absolutely continuous with
respect to the Lebesgue measure, and its density is given by 
\begin{equation}\label{eq: formula for p^eta_l}
p_{l}^{\eta}(t,x,y)\triangleq\int_{M_{x,y}}\eta(u)K(u,x)\rho_{t}(u)m_{x,y}(du).
\end{equation}
where $\rho_t$ is the density of $U_t$ and $K$ is given by \eqref{kernel in disintegration}.
\end{prop}

To obtain a sharp lower estimate on $p_{l}^{\eta}(t,x,y)$ from 
formula \eqref{eq: formula for p^eta_l} and the lower estimate of $\rho_t(u)$ given by \eqref{signature density lower bound}, one needs to estimate the volume form $m_{x,y}$ precisely. For this purpose, we resort to a change of variables 
$S_{x,y}$ from the "bridge space" $M_{x,y}$ to the "loop space" $M_{x,x}$  introduced in \cite{KS87}. 
The construction of this map $S_{x,y}$ is based on the simple idea that, in order to transform an arbitrary loop $\alpha$ from $x$ to $x$ into a path from $x$ to $y$, we just concatenate the loop $\alpha$ to a fixed path from $x$ to $y$. However, this idea does not project to a map from $M_{x,y}$ to $M_{x,x}$ in a simple way, and one needs to use the function $\Psi_l$ defined in Lemma~\ref{lem: the Psi function} to make it work at the level of Taylor approximation. We summarize the technical formulation of this map $S_{x,y}$ in the following lemma, which was proved by \cite{KS87} in the diffusion case (i.e. $H=1/2$). Thanks to our Theorem \ref{thm:equiv-distances}, the lemma (in particular the estimate (\ref{eq: estimating S_xy}) below) holds in exactly the same way in the fractional Brownian context for $H>1/4$.

\begin{lem}\label{lem: construction of S_xy}
Recall that the function $\Psi_{l}$ is defined in Lemma~\ref{lem: the Psi function}. We define the operation $\times$ to be the multiplication induced from $G^{(l)}$ through the exponential map, namely
\begin{equation}\label{multiplication on Lie algebra}
v\times u\triangleq\log( \exp(v)\otimes \exp(u)),\ \ \ v,u\in\mathfrak{g}^{(l)}.
\end{equation}
Then the following statements hold true:

\vspace{2mm}\noindent
(i) There exist $\varepsilon,\rho_{1}>0$ and $\rho_{2}\in(0,r),$
such that for any given $x\in\mathbb{R}^{N}$ and $h\in\bar{{\cal H}}$
with $\|h\|_{\bar{{\cal H}}}<\rho_{1}$, the map 
\[
v\mapsto\tilde{\Psi}_{x,h}(v)\triangleq\Psi_{l}(v\times u,x,y-x-F_{l}(v\times u,x)),
\]
where $y\triangleq\Phi_{1}(x;h)$ and $u\triangleq\log S_{l}(h)$,
defines a diffeomorphism from an open neighbourhood $V_{x,h}\subseteq\mathfrak{g}^{(l)}$
of $0$ containing the ball $\{v\in\mathfrak{g}^{(l)}:\|v\|_{{\rm HS}}<\varepsilon\}$
onto $W\triangleq\{w\in\mathfrak{g}^{(l)}:\|w\|_{{\rm HS}}<\rho_{2}\}$,
such that 
\[
v\in V_{x,h}\cap M_{x,x}\iff w\triangleq\tilde{\Psi}_{x,h}(v)\in W\cap M_{x,y}.
\]
(ii) Given $x,y\in\mathbb{R}^{N}$ with $d(x,y)<\rho_{1}/2$, we choose
$h\in\Pi_{x,y}$ satisfying 
\[
d(x,y)\leq\|h\|_{\bar{{\cal H}}}\leq2d(x,y)<\rho_{1}
\]
and define 
\begin{align}\label{def of Sxy}
S_{x,y}\triangleq\tilde{\Psi}_{x,h}^{-1}\rrn _{W\cap M_{x,y}}:W\cap M_{x,y}\rightarrow V_{x,h}\cap M_{x,x}.
\end{align}
Then there exists a constant $\Lambda>0,$ such that
\begin{align}\label{compare vol mxx and mxy}
\frac{1}{\Lambda}\cdot m_{x,x}(\cdot)\leq m_{x,y}\circ S_{x,y}^{-1}(\cdot)\leq\Lambda\cdot m_{x,x}(\cdot)\ \ \ {\rm on}\ V_{x,h}\cap M_{x,x},
\end{align}
and 
\begin{align}\label{eq: estimating S_xy}
\frac{1}{\Lambda}\cdot(\|v\|_{{\rm CC}}+d(x,y))\leq\|S_{x,y}^{-1}(v)\|_{{\rm CC}}\leq\Lambda\cdot(\|v\|_{{\rm CC}}+d(x,y))
\end{align}
for any $v\in V_{x,h}\cap M_{x,x}.$
\end{lem}

Now we apply the change of variables involving $S_{x,y}$ in Lemma \ref{lem: construction of S_xy}  to establish a lower estimate of the density $p_l^\eta(t,x,y)$ in terms of the measure $m_{x,x}$ that does not depend on $y$.

\begin{lem}\label{lem: lower estimate in terms of m_xx}
\label{lem: lower estimate in terms of m_xx}Let $p_l^\eta(t,x,y)$ be the density defined by \eqref{eq: formula for p^eta_l}, and recall that the exponent $\nu$ is defined by $\nu\triangleq\sum_{k=1}^{l}k\dim{\cal L}_{k}.$ Then there exist constants
$C,\tau>0$, such that for all $x,y,t$ with $d(x,y)\leq t^{H}$
and $0<t<\tau$, we have 
\begin{equation}
p_{l}^{\eta}(t,x,y)\geq Ct^{-H\nu}m_{xx}\lp \{v\in M_{x,x}:\|v\|_{\textsc{CC}}\leq t^{H}\}\rp ,\label{eq: lower estimate in terms of m_xx}
\end{equation}
where $m_{x,x,}$ is the volume form on $M_{x,x}$ given by \eqref{Mxy}
\end{lem}
\begin{proof} Lemma \ref{lem: construction of S_xy} asserts that there exists $\rho_1>0$, such that if $d(x,y)<\rho_1/2$, then $S_{x,y}$ given by \eqref{def of Sxy} defines a change of variables (i.e. a diffeomorphism) for (\ref{eq: formula for p^eta_l}). Specifically we have
\begin{align*}
p_{l}^{\eta}(t,x,y) & \geq\int_{M_{x,y}\cap W}\eta(u)K(u,x)\rho_{t}(u)m_{x,y}(du)\\
 & =\int_{M_{x,x}\cap V_{x,h}}\eta(S_{x,y}^{-1}v)K(S_{x,y}^{-1}v,x)\rho_{t}(S_{x,y}^{-1}v)
 \, m_{x,y}\circ S_{x,y}^{-1}(dv).
 \end{align*}
In addition, since $V_{x,h}$ contains the ball $\{v\in\mathfrak{g}^{(l)}:\|v\|_{\mathrm{HS}}<\varepsilon\}$, owing to relation \eqref{compare vol mxx and mxy} and thanks to the fact that $K$ defined by \eqref{kernel in disintegration} is bounded below, we obtain
 \begin{align*}
 p_l^\eta(t,x,y) \geq C_{H,V,l}\int_{M_{x,x}\cap\{v\in\mathfrak{g}^{(l)}:\|v\|_{\textsc{HS}}<\varepsilon\}}\rho_{t}(S_{x,y}^{-1}v) \, m_{x,x}(dv).
\end{align*}
Now choose $\tau<(\rho_{1}/2)^{\frac{1}{H}}$ to be such that 
\[
0<t<\tau\implies\{v\in\mathfrak{g}^{(l)}:\|v\|_{\textsc{CC}}\leq t^{H}\}\subseteq\lcl  v\in\mathfrak{g}^{(l)}:\|v\|_{\textsc{HS}}<\varepsilon\rcl  .
\]
We will thus lower bound $p_l^\eta(t,x,y)$ as follows
 \begin{align}\label{lower bound p_l^eta 1}
 p_l^\eta(t,x,y) \geq C_{H,V,l}\int_{M_{x,x}\cap\{v\in\mathfrak{g}^{(l)}:\|v\|_{\textsc{CC}}\leq t^{H}\}}
 \rho_{t}(S_{x,y}^{-1}v)m_{x,x}(dv).
\end{align}

Next, according to the second inequality of (\ref{eq: estimating S_xy}),
if $d(x,y)\leq t^{H}$ and $t<\tau$ (so that $d(x,y)<\rho_{1}/2$),
then 
\[
\|S_{x,y}^{-1}v\|_{\textsc{CC}}\leq2Ct^{H},
\]
provided that $v\in M_{x,x}$ with $\|v\|_{\textsc{CC}}\leq t^{H}.$
For such $x,y,t,v$, by Proposition \ref{prop: step one} we have 
\[
\rho_{t}(S_{x,y}^{-1}v)\geq\beta_{2C}t^{-H\nu}.
\]
Plugging this inequality into \eqref{lower bound p_l^eta 1}, we arrive at 
\[
p_{l}^{\eta}(t,x,y)\geq C_{H,V,l}\beta_{2C}t^{-H\nu}m_{x,x}\lp \{v\in M_{x,x}:\|v\|_{\textsc{CC}}\leq t^{H}\}\rp .
\]

\end{proof}
The next lemma relates the measure $m_{x,x}$ with the volume of the ball $B_d(x,t^H)$ defined with respect to the control distance function. 
\begin{lem}\label{lem: estimate vol m_xx}
Let $M_{x,x}$ be the set defined by \eqref{Mxy} and recall that $m_{x,x}$ is the volume measure on $M_{x,x}$.
There exist constants $C,\tau>0,$ such that 
\begin{align}\label{estimate vol m_xx}
t^{-H\nu}m_{x,x}\lp \{v\in M_{x,x}:\|v\|_{\textsc{CC}}\leq t^{H}\}\rp \geq\frac{C}{|B_{d}(x,t^{H})|}
\end{align}
for all $x\in\mathbb{R}^{N}$ and $0<t<\tau$.
\end{lem}

\begin{proof}
We know from \cite[Lemma 3.31]{KS87} that for any $\beta\in(0,1)$, there exists a constant $C_\beta$ such that
\begin{align}\label{eq:EstMDif}
t^{-\nu/2}m_{xx}\big(\big\{ v\in M_{x,x}:\|v\|_{{\rm CC}}\leq t^{1/2}\big\}\big)\geq\frac{C_{\beta}}{|B_{d_{{\rm BM}}}(x,\beta t^{1/2})|}
\end{align}
for all $(t,x)\in(0,1]\times\mathbb{R}^{N}$, where $d_\mathrm{BM}$ is the control distance function for the diffusion case. On the other hand, according to Theorem \ref{thm:equiv-distances}, there exist constants $C_H$ and $\delta$, such that \[
d(y,x)\leq C_H d_{{\rm BM}}(y,x)
\]when $|y-x|<\delta$. If we choose $\beta = \frac{1}{C_H}$, when $t$ is small we have \[
B_{d_{{\rm BM}}}(x,\beta t^{1/2})\subseteq B_{d}(x,t^{1/2})\subseteq\{y:|y-x|<\delta\}
\]and thus the right hand side of (\ref{eq:EstMDif}) is further bounded below by $\frac{C_\beta}{|B_d(x,t^{1/2})|}$. The desired inequality (\ref{estimate vol m_xx}) follows by changing $t\mapsto t^{2H}$.

\end{proof}

Summarizing the contents of Lemma \ref{lem: lower estimate in terms of m_xx} and Lemma \ref{lem: estimate vol m_xx},  we have obtained the following lower bound on $p_l^\eta(t,x,y)$, which finishes the second step of the main strategy.

\begin{cor}Let $p_l^\eta(t,x,y)$ be the density given by \eqref{eq: formula for p^eta_l}, and recall the notations of Lemma \ref{lem: lower estimate in terms of m_xx}. Then there exist constants  $C,\tau>0$
depending only on $H,l$ and the vector fields, such that 
\begin{equation}\label{eq: lower estimate for approximating density}
p_{l}^{\eta}(t,x,y)\geq\frac{C}{|B_{d}(x,t^{H})|}
\end{equation}
for all $x,y,t$ satisfying $d(x,y)\leq t^{H}$ and $0<t<\tau$.
\end{cor}

\subsection{Step three: comparing approximating and actual densities}\label{sec: step three}

The last step towards the proof of Theorem \ref{thm: local lower estimate}
will be to show that the approximating density $p_{l}^{\eta}(t,x,y)$
and the actual density $p(t,x,y)$ of $X_t^x$ are close to each other when $t$
is small. For this part, we combine the Fourier transform approach
developed in \cite{KS87} with general estimates for Gaussian rough differential
equations. As before, we assume that $l\geq l_0$.

Recall that the Fourier transform of a function $f(y)$ on $\mathbb{R}^{N}$
is defined by 
\[
\mathcal{F}f (\xi)=\hat{f}(\xi)\triangleq\int_{\mathbb{R}^{N}}f(y){\rm e}^{2\pi i\langle\xi,y\rangle} \, dy,\ \ \ \xi\in\mathbb{R}^{N}.
\]
In the sequel we will consider the Fourier transform $\hat{p}(t,x,\xi)$ (respectively, $\hat{p}_{l}^{\eta}(t,x,\xi)$)
of the density $p(t,x,y)$ (respectively, $p_{l}^{\eta}(t,x,y)$)
with respect to the $y$-variable. We will invoke the following trivial bound on $p-p_{l}$ in terms of $\hat{p}$ and $\hat{p}_{l}^{\eta}$:
\begin{equation}
|p(t,x,y)-p_{l}^{\eta}(t,x,y)|\leq\int_{\mathbb{R}^{N}}\lln \hat{p}(t,x,\xi)-\hat{p}_{l}^{\eta}(t,x,\xi)\rrn d\xi.\label{eq: Fourier estimate}
\end{equation}
Therefore our aim in this section will be to estimate the right hand side of \eqref{eq: Fourier estimate} by considering two regions $\{|\xi|\leq R\}$ and $\{|\xi|>R\}$ separately in the integral, where $R$ is some large number to
be chosen later on.

\subsubsection{Integrating relation \eqref{eq: Fourier estimate} in a neighborhood of the origin}
\label{sec: integrate-small}

We first integrate our Fourier variable $\xi$ in \eqref{eq: Fourier estimate} over the region $\{|\xi|\leq R\}$.
In this case, we make use of a
tail estimate for the error of the Taylor approximation of $X_{t}^{x}$ which is provided below. 
\begin{lem}
\label{lem: tail estimate for Taylor error}
Let $X_{t}^{x}$
be the solution to the SDE (\ref{eq: hypoelliptic SDE}), and consider its approximation $X_{l}(t,x)$ of order $l\ge \max\{l_{0},H^{-1}\}$ as defined by \eqref{eq:def-X-l}. Fix $\bar{l}\in(l,l+1)$.
Then there exist constants $C_{1},C_{2}$ depending only on $H,l$ and
the vector fields, such that for all $t\in(0,1]$ and $x,y\in\R^{N}$ we have
\begin{align}\label{eq: tail estimate of difference}
\mathbb{P}\lp |X_{t}^{x}-X_{l}(t,x)|\geq\lambda\rp \leq C_{1}\exp\lp -\frac{C_{2}\lambda^{2/\bar{l}}}{t^{2H}}\rp ,\ \ \ \text{ for all }\lambda>0.
\end{align}
\end{lem}
\begin{proof}
According to \cite[Corollary 10.15]{FV10}, we have the
following pathwise estimate
\begin{align}\label{eq: rough path estimate Euler}
|X(t,x)-X_{l}(t,x)|\leq C\cdot\|{\bf B}\|_{p-{\rm var};[0,t]}^{\bar{l}},
\end{align}
where $C=C_{H,V,l}>0$ and  ${\bf B}$ is the rough path
lifting of $B$ alluded to in Proposition \ref{prop:fbm-rough-path}. In equation \eqref{eq: rough path estimate Euler}, the parameter $p$ is any number greater that $1/H$ and the $p$-variation norm is defined with respect to the CC-norm. It follows
from \eqref{eq: rough path estimate Euler} that for any $\lambda>0$ and $\eta>0,$
we have 
\begin{align*}
  \mathbb{P}\lp |X_{t}^{x}-X_{l}(t,x)|\geq\lambda\rp 
 \leq\mathbb{P}\lp \|{\bf B}\|_{p-{\rm var};[0,t]}^{\bar{l}}\geq\lambda/C\rp .
 \end{align*}
 In addition, the fBm signature satisfies the identity in law
 $$
 ({\bf B}_s)_{0\leq s\leq t}
 \stackrel{d}{=}
 \lp \delta_{t^H}\circ{\bf B}_{s/t}\rp _{0\leq s\leq t}.
 $$
 Owing to the scaling properties of the CC-norm, we thus get that for an arbitrary $\zeta>0$ we have
 \begin{align}
 \mathbb{P}(|X^x_t-X_l(t,x)|\geq \lambda)& 
 \leq 
 \mathbb{P}\lp \|{\bf B}\|_{p-{\rm var};[0,1]}\geq\frac{(\lambda/C)^{1/\bar{l}}}{t^{H}}\rp 
 \nonumber \\
 & \leq
 \exp\lp -\frac{\zeta(\lambda/C)^{2/\bar{l}}}{t^{2H}}\rp \cdot
 \mathbb{E}\lc {\rm e}^{\zeta\|{\bf B}\|_{p-{\rm var};[0,1]}^{2}}\rc  .\label{eq: tail estimate for Taylor error}
\end{align}
According to a Fernique type estimate for the fractional Brownian rough path (cf. \cite[Theorem 15.33]{FV10}), we know that there exists $\zeta=\zeta_{H}>0$ such that
\[
\mathbb{E}\lc {\rm e}^{\zeta\|{\bf B}\|_{p-{\rm var};[0,1]}^{2}}\rc  <\infty.
\]
Plugging this inequality into \eqref{eq: tail estimate for Taylor error}, our conclusion \eqref{eq: tail estimate of difference} is easily obtained. 
\end{proof}

We are now ready to derive a Fourier transform estimate for small values of $\xi$.
\begin{lem}\label{lem: bounded region part}
Under the same notation as in Lemma \ref{lem: tail estimate for Taylor error}, let $p(t,x,\cdot)$ be the density of the random variable $X_t^x$ and let $p_l^\eta(t,x,\cdot)$ be the approximating density given by \eqref{eq: formula for p^eta_l}. Then their Fourier transforms satisfy the following inequality over the region $\{|\xi|\leq R\}$,
\begin{align}\label{eq: bounded region part}
 \lln \hat{p}(t,x,\xi)-\hat{p}_{l}^{\eta}(t,x,\xi)\rrn \leq C_{H,V,l}(1+|\xi|)t^{H\bar{l}},
\end{align}provided that $t<\tau_1$ for some constant $\tau_1$ depending on $H,l$ and the vector fields.
\end{lem}
\begin{proof}Notice that according to our definition \eqref{eq: formula for p^eta_l} of $p^\eta_l$, we have
\begin{align*}
\hat{p}(t,x,\xi)=\mathbb{E}\lc e^{2\pi i\langle \xi, X_t^x\rangle}\rc  ,\quad\text{and}\quad \hat{p}_l^\eta(t,x,\xi)=\mathbb{E}\lc \eta(U_t)e^{2\pi i\langle \xi, X_l(t,x)\rangle}\rc  .
\end{align*}
Hence it is easily seen that
\begin{eqnarray}
  \lln \hat{p}(t,x,\xi)-\hat{p}_{l}^{\eta}(t,x,\xi)\rrn  
 &\leq&
 \mathbb{E}\lc \lln {\rm e}^{-2\pi i\langle\xi,X_{t}^{x}\rangle}-{\rm e}^{-2\pi i\langle\xi,X_{l}(t,x)\rangle}\rrn \rc  +\mathbb{E}[1-\eta(U_{t})]\nonumber\\
 & \leq&
 2\pi|\xi|\cdot\mathbb{E}\lc |X_{t}^{x}-X_{l}(t,x)|\rc  +\mathbb{E}[1-\eta(U_{t})].\label{Fourier difference bound mid step1}
\end{eqnarray}
Now in order to bound the right hand-side of \eqref{Fourier difference bound mid step1}, we first invoke Lemma \ref{lem: tail estimate for Taylor error}. This yields
\begin{eqnarray}
  \mathbb{E}\lc |X_{t}^{x}-X_{l}(t,x)|\rc  
 & =&
 \int_{0}^{\infty}\mathbb{P}\lp |X_{t}^{x}-X_{l}(t,x)|\geq\lambda\rp d\lambda\nonumber\\
 & \leq& 
 C_{1}\int_{0}^{\infty}\exp\lp -\frac{C_{2}\lambda^{2/\bar{l}}}{t^{2H}}\rp d\lambda
  =
 C_{3}t^{H\bar{l}}.\label{Fourier difference bound mid step2}
\end{eqnarray}
On the other hand, using a similar argument to the proof of Lemma
\ref{lem: tail estimate for Taylor error}, we see that there exists a strictly positive exponent $\al_{H,l}$ such that
\begin{align}\label{Fourier difference bound mid step3}
\mathbb{E}[1-\eta(U_{t})]\leq\mathbb{P}\lp \|U_{t}\|_{\textsc{HS}}\geq\frac{r}{2}\rp \leq C_{4}\cdot\exp\big(-\frac{C_{5}}{t^{\alpha_{H,l}}}\big).
\end{align}
By taking $t$ small enough, we can make the right hand-side of \eqref{Fourier difference bound mid step3} smaller than $C_{6}t^{H\bar{l}}$.  Hence there exists $\tau_1>0$ such that if $t\leq \tau_1$ we have
\begin{align}\label{Fourier difference bound mid step4}
\mathbb{E}[1-\eta(U_{t})]\leq C_{6}t^{H\bar{l}}.
\end{align}
Now combining \eqref{Fourier difference bound mid step2} and \eqref{Fourier difference bound mid step4}, we easily get our conclusion \eqref{eq: bounded region part}.
\end{proof}

\subsubsection{Integrating relation \eqref{eq: Fourier estimate} for large Fourier modes}
\label{sec: integrate-large}

We now integrate the Fourier variable $\xi$ over the region $\{|\xi|>R\}$.
In this case, we make use of certain upper estimates
for $p(t,x,y)$ and $p_{l}^{\eta}(t,x,y)$. We start with a bound on the density of $X_{t}^{x}$ which is also of independent
interest. The main ingredients of the proof are basically known in
the literature, but to our best knowledge the result (for the hypoelliptic
case) has not been formulated elsewhere.
\begin{prop}
\label{prop: Gaussian upper bound for p}Let $p(t,x,y)$ be the density of the random variable $X_t^x$. As in Lem\-ma~\ref{lem: bounded region part} we assume that the uniform hypoellipticity condition  \eqref{eq:unif-hypo-assumption} is satisfied. Then for each $n\geq1,$
there exist constants $C_{1,n},C_{2,n},\nu_{n}>0$ depending on $n,H$
and the vector fields such that 
\begin{align}\label{general density upper bound for hypoellitpic SDE}
|\partial_{y}^{n}p(t,x,y)|\leq C_{1,n}t^{-\nu_{n}}\exp\lp -\frac{C_{2,n}|y-x|^{2\wedge(2H+1)}}{t^{2H}}\rp ,
\end{align}
for all $(t,x,y)\in(0,1]\times\mathbb{R}^{N}\times\mathbb{R}^{N}$, where $\partial_y^n$ denotes the $n$-th order derivative operator with respect to the $y$ variable.
\end{prop}

\begin{proof}
Elaborating on the integration by parts invoked for example in \cite[Relation (24)]{BNOT16}, there exist exponents $\alpha, \beta, p,q>1$ such that 
\begin{align}\label{IBP density bound}
|\partial_{y}^{n}p(t,x,y)|\leq C_{1,n}\mathbb{P}(|X_t^x-x|\geq|y-x|)^{\frac{1}{2}}\cdot\|\gamma^{-1}_{X_t^x}\|_{\alpha,p}^\alpha\cdot\|\mathbf{D}X_t^x\|_{\beta,q}^\beta,
\end{align}
where  $\gamma_{X_t^x}$ denotes the Malliavin covariance matrix and $\|\cdot\|_{k, p}$ denotes the Gaussian-Sobolev norm.  Then with \eqref{IBP density bound} in hand, we proceed in the following way:

\vspace{2mm}

\noindent (i) An exponential tail estimate for $X_t^x$ yield the exponential term in \eqref{general density upper bound for hypoellitpic SDE}. This step is achieved as in \cite[Relation (25)]{BNOT16}.\\
(ii) The Malliavin derivatives of $X_t^x$ are estimated as in \cite[Lemma 3.5 (1)]{BOZ15}. This produces some positive powers of $t$ in \eqref{general density upper bound for hypoellitpic SDE}.\\
(iii) The inverse of the Malliavin covariance matrix is bounded as in \cite[Lemma 3.5 (2)]{BOZ15}. It gives some negative powers of $t$ in \eqref{general density upper bound for hypoellitpic SDE}.

\vspace{2mm}

\noindent For the sake of conciseness, we will not detail the steps outlined as above. We refer the reader to \cite{BNOT16, BOZ15} for the details. 
\end{proof}
The following lemma parallels Proposition \ref{prop: Gaussian upper bound for p} for the approximating process $X_l(t,x)$. 
\begin{lem}
\label{lem: upper estimate for approximating density}Assume the same hypothesis as in Proposition \ref{prop: Gaussian upper bound for p}. Recall that the approximating density $p_l^\eta$ is defined by \eqref{eq: formula for p^eta_l}.  { Fix $l\geq l_0$}. Then
for each $n\geq1$ there exist constants $C_n=C_n(H,l)$ and $\gamma_n=\gamma_n(H,l_0)$ such that for all $(t,x)\in(0,1]\times\mathbb{R}^{N}$ the following bound holds true:
\begin{align}\label{eq: upper bound approximating density}
\|\partial_{y}^{n}p_{l}^{\eta}(t,x,\cdot)\|_{C_{b}^{n}(\mathbb{R}^{N})}\leq C_{n}\cdot t^{-\gamma_{n}}.
\end{align}
Moreover, the function $\partial_{y}^{n}p_{l}^{\eta}(t,x,\cdot)$ is compactly supported in $\R^{N}$.
\end{lem}

\begin{proof}
Recall that for a
differentiable random vector $Z=(Z^{1},\ldots,Z^{n})$ in the sense
of Malliavin, we use the notation $\gamma_{Z}\triangleq(\langle DZ^{i},DZ^{j}\rangle_{\bar{{\cal H}}}){}_{1\leq i,j\leq n}$
to denote its Malliavin covariance matrix. By the definition \eqref{eq:def-X-l} of
$X_{l}(t,x),$ it is immediate that 
\begin{align}\label{equ: Malliavin matrix X_l}
\gamma_{X_{l}(t,x)}=JF_{l}(U_{t}^{(l)},x)\cdot \gamma_{U_{t}^{(l)}}\cdot JF_{l}(U_{t}^{(l)},x)^{*}.
\end{align}
It follows from \cite[Theorem 3.3]{BFO19} and Lemma \ref{lem: nondegeneracy of F} that $\gamma^{-1}_{X_l(t,x)}\in L^q$ for all $q>1$. Now  the uniform upper bound for the derivatives
of $p_{l}^{\eta}(t,x,y)$ follows from the same lines as in the proof
of Proposition \ref{prop: Gaussian upper bound for p} (with the same three
main ingredients (i),(ii),(iii)), based on the integration by parts formula. 
\end{proof}

\begin{rem}We must point out that the exponent $\gamma_n$ in \eqref{eq: upper bound approximating density} depends on $l_0$ but \textit{not} on $l$. This subtle technical point is critical for us, and its proof requires a non-trivial amount of analysis. For the sake of conciseness, we refer the reader to \cite[Section 5.3 IV]{GOT19} for a detailed discussion on this matter. 
\end{rem}

We now  return to the Fourier estimate (\ref{eq: Fourier estimate}) for
the region $|\xi|>R$.  In particular, we have the following result.

\begin{lem}\label{lem: large xi part}
Using the same notation and hypothesis  as in Lemma \ref{lem: bounded region part}, the Fourier transforms $\hat{p}$ and $\hat{p}_l^\eta$ are such that for all $|\xi| > R$ we have
\begin{equation}
|\xi|^{N+2}\lp \lln \widehat{p}(t,x,\xi)\rrn +\lln \widehat{p}_{l}^{\eta}(t,x,\xi)\rrn \rp \leq C\cdot t^{-\mu},
\label{eq: outer ball part}
\end{equation}
for some strictly positive constants $C=C_{N,H,V,l}$ and $\mu=\mu_{N,H,V,l_{0}}$.
\end{lem}

\begin{proof}According to standard compatibility rules between Fourier transform and differentiation, we have (recall that $\mathcal{F}f$ and $\hat{f}$ are both used to designate the Fourier transform of a function $f$):
\begin{align*}
&|\xi|^{N+2}(|\hat{p}(t,x,\xi)|+|\hat{p}^\eta_l(t,x,\xi)|)\\
&\leq C_N\lp |\mathcal{F}(\partial_y^{N+2}p(t,x,y)|+|\mathcal{F}(\partial_y^{N+2}p^\eta_l(t,x,y)|\rp .
\end{align*}
Plugging \eqref{general density upper bound for hypoellitpic SDE} and \eqref{eq: upper bound approximating density} into this relation and using the fact that $\partial_y^{N+2}p^\eta_l(t,x,\cdot)$ is compactly supported, our claim \eqref{eq: outer ball part} is easily proved.
\end{proof}

\subsubsection{Comparison of the densities}
\label{sec: com-den}

Combining the previous preliminary results on Fourier transforms we get the following uniform bound on the difference $p-p^\eta_l$.
\begin{prop}
We still keep the same notation and assumptions of Lemma \ref{lem: bounded region part}. Then there exists $\tau>0$ such that for all $t\leq \tau$ and $x,y\in\mr^N$ we have
\begin{align}\label{eq: comparison estimate}
|p(t,x,y)-p_l^\eta(t,x,y)|\leq C_{H,V,l} \,t.
\end{align}
\end{prop}

\begin{proof}
Thanks to \eqref{eq: Fourier estimate} we can write 
\begin{align*}
|p(t,x,y)-p_{l}^{\eta}(t,x,y)|\leq\lp \int_{|\xi|\leq\mr}+\int_{|\xi|>\mr}\rp \lln \hat{p}(t,x,\xi)-\hat{p}_{l}^{\eta}(t,x,\xi)\rrn d\xi.
\end{align*}
Next we invoke the bounds (\ref{eq: bounded region part}) and (\ref{eq: outer ball part}), which allows us to write
\begin{align*}
  \lln p(t,x,y)-p_{l}^{\eta}(t,x,y)\rrn 
  \leq C_{1}\lp R^{N+1}t^{H\bar{l}}+t^{-\mu}\int_{|\xi|>R}|\xi|^{-N-2}d\xi\rp ,\ \ \ t\in [0,1],
\end{align*}
where we recall that $\bar{l}$ is a fixed number in $[l,l+1]$ introduced in Lemma \ref{lem: tail estimate for Taylor error}. Now an elementary change of variable yields
\begin{align}
 & \lln p(t,x,y)-p_{l}^{\eta}(t,x,y)\rrn \nonumber\\
 & \leq C_{1}\lp R^{N+1}t^{H\bar{l}}+t^{-\mu}R^{-1}\int_{|\xi|\geq1}|\xi|^{-(N+1)}d\xi\rp \nonumber\\
 & \leq C_{2}\lp R^{N+1}t^{H\bar{l}}+t^{-\mu}R^{-1}\rp .\label{eq: bound difference of densities mid step1}
\end{align}
We can easily optimize expression \eqref{eq: bound difference of densities mid step1} with respect to $R$ by choosing  $R=t^{-(\mu+1)}$ . It follows that
\begin{align}\label{eq: bound difference of densities mid step 2}
\lln p(t,x,y)-p_{l}^{\eta}(t,x,y)\rrn \leq C_{2}t^{-(N+1)(\mu+1)+H\overline{l}}+t,
\end{align}
for all  $t\in(0,1]$.  In addition, recall that a crucial point in our approach is that the exponent $\mu$ in \eqref{eq: bound difference of densities mid step 2} does not depend on $l$. Therefore  we can choose
$l\geq l_{0}$ large enough, so that 
\[
-(N+1)(\mu+1)+H\overline{l}\geq1.
\]
For this value of $l$, the upper bound \eqref{eq: comparison estimate} is easily deduced from \eqref{eq: bound difference of densities mid step 2}.


\end{proof}

\subsection{Completing the proof of Theorem \ref{thm: local lower estimate}}

Finally, we are in a position to complete the proof of Theorem \ref{thm: local lower estimate}.
Indeed, recall from (\ref{eq: lower estimate for approximating density})
and (\ref{eq: comparison estimate}) that for $x$ and $y$ such that $d(x,y)\leq t^H$ and $t<\tau$, we have
\begin{align}\label{f1}
p^\eta(t,x,y)\geq \frac{C_{1}}{|B_d(x,t^H)|},\quad\text{and}\quad |p(t,x,y)-p^\eta_l(t,x,y)|\leq C_{H,V,l}t.
\end{align}
In addition, owing to \eqref{eq:local-comparison}, for small $t$ we have
\begin{equation}\label{f2}
\frac{1}{|B_d(x,t^H)|}\geq\frac{C}{t^{HN/l_0}}.
\end{equation}
Putting together \eqref{f1} and \eqref{f2}, it is easily seen that when $t$ is small enough we have
\[
p(t,x,y)\geq\frac{C_{2}}{|B_{d}(x,t^{H})|}.
\]
Therefore, the proof of  Theorem \ref{thm: local lower estimate} is  complete.


\begin{thebibliography}{99}

\bibitem{BFO19}F. Baudoin, Q. Feng and C. Ouyang, Density of the signature process of fBM, preprint, \textit{arXiv:1904.09384}, 2019.

\bibitem{BH07}F. Baudoin and M. Hairer, A version of H\"ormander's theorem for the fractional Brownian motion, \textit{Probab. Theory Related Fields} 139 (3-4) (2007) 373--395.

\bibitem{BNOT16}F. Baudoin, E. Nualart, C. Ouyang and S. Tindel, On probability laws of solutions to differential systems driven by a fractional Brownian motion, \textit{Ann. Probab.} 44 (4) (2016) 2554--2590.

\bibitem{BOZ15}F. Baudoin, C. Ouyang and X. Zhang, Varadhan estimates for rough differential equations driven by fractional Brownian motions, \textit{Stochastic Process. Appl.} 125 (2015) 634--652.

\bibitem{BOZ16}F. Baudoin, C. Ouyang and X. Zhang, Smoothing effect of rough differential equations driven by fractional Brownian motions, 
\textit{Ann. Inst. H. Poincar\'e Probab. Statist.} 52 (1) (2016) 412--428.

\bibitem{BW}
G. Ben Arous, J. Wang,
Very rare events for diffusion processes in short time,
\textit{arXiv:1901.10025}.

\bibitem{Bis84}
J. Bismut, \textit{Large deviations and the Malliavin calculus}, in "Progress in Mathematics" Vol. 45, Birkh\"auser, 1984.

\bibitem{CF10} T. Cass and P. Friz, Densities for rough differential equations under H\"ormander condition, \textit{Ann. of Math.} 171 (3) (2010) 2115--2141.

\bibitem{CHLT15} T. Cass, M. Hairer, C. Litterer and S. Tindel, Smoothness of the density for  solutions to Gaussian rough differential equations, \textit{Ann. Probab.} 43 (1) (2015) 188--239.

\bibitem{CQ}
L. Coutin, Z. Qian,
Stochastic analysis, rough path analysis and fractional Brownian motions. 
\textit{Probab. Theory Related Fields} 122 (2002), no. 1, 108--140.

\bibitem{Da07} A. M. Davie,
Differential Equations Driven by Rough Paths: An Approach via Discrete Approximation.
\newblock {\it Applied Mathematics Research Express. AMRX}, 2 (2007).

\bibitem{DU97}L. Decreusefond and A. \"Ust\"unel, Stochastic analysis of the fractional Brownian motion, \textit{Potential Anal.} 10 (1997) 177--214.

\bibitem{FH14}P. Friz and M. Hairer, \textit{A course on rough paths with an introduction to regularity structures}, Springer, 2014.

\bibitem{FV06}P. Friz and N. Victoir, A variation embedding theorem and applications, \textit{J. Funct. Anal.} 239 (2006) 631--637.

\bibitem{FV10}P. Friz and N. Victoir, \textit{Multidimensional stochastic processes as rough paths: theory and applications}, Cambridge University Press, 2010.

\bibitem{GOT19}X. Geng, C. Ouyang and S. Tindel, Precise local estimates for hypoelliptic differential equations
driven by fractional Brownian motion, \textit{arXiv:1907.00171}, 2019.

\bibitem{GOT20}X. Geng, C. Ouyang and S. Tindel, Precise local estimates for differential equations
driven by fractional Brownian motion: elliptic case, preprint, 2020. 


\bibitem{GOT}
B. Gess, C. Ouyang, S. Tindel, 
Density bounds for solutions to differential equations driven by Gaussian rough paths. 
\textit{J. Theoret. Probab.} 33 (2020), no. 2, 611-648.

\bibitem{KMS93}A. Kilbas, O. Marichev and S. Samko, \textit{Fractional integrals and derivatives: theory and applications}, Gordon and Breach, Amsterdam, 1993.

\bibitem{KS84}S. Kusuoka and D. Stroock, 
Applications of the Malliavin calculus. I. 
Stochastic analysis (Katata/Kyoto, 1982), 271--306, North-Holland Math. Library, 32, North-Holland, Amsterdam, 1984.

\bibitem{KS84b}S. Kusuoka and D. Stroock, 
Applications of the Malliavin calculus. II. 
\textit{J. Fac. Sci. Univ. Tokyo} Sect. IA Math. 32 (1985), no. 1, 1-76.

\bibitem{KS87}S. Kusuoka and D. Stroock, Applications of the Malliavin calculus, Part III, \textit{J. Fac. Sci. Univ. Tokyo} 34 (1987) 391--442.

\bibitem{IP}
Y. Inahama, B. Pei,
Positivity of the density for rough differential equations.
\textit{arXiv:2006.09631}, 2020.

\bibitem{Lyons94} T. Lyons, Differential equations driven by rough signals (I): an extension of an inequality of L.C.Young, \textit{Math. Res. Lett.} 1 (1994) 451--464.

\bibitem{Lyons98}T. Lyons, Differential equations driven by rough signals, \textit{Rev. Mat. Iberoam.} 14 (2) (1998) 215--310.

\bibitem{Ma78}
P. Malliavin,
Stochastic calculus of variation and hypoelliptic operators. \textit{Proceedings of the International Symposium on Stochastic Differential Equations} (\textit{Res. Inst. Math. Sci.}, Kyoto Univ., Kyoto, 1976), pp. 195--263, Wiley, 1978. 

\bibitem{Montgomery02}
R. Montgomery,  
\textit{A tour of subriemannian geometries, their geodesics and applications}, 
American Mathematical Society, Providence,  2002.

\bibitem{Milnor97}J.W. Milnor, \textit{Topology from the differentiable viewpoint}, Princeton University Press, Princeton, 1997.

\bibitem{Nualart06}D. Nualart, \textit{The Malliavin calculus and related topics}, Springer-Verlag, 2006.

\bibitem{NNRT}
A. Neuenkirch, I. Nourdin, A. R\"o\ss ler, S. Tindel,
Trees and asymptotic developments for fractional diffusion processes.
{\it  Ann. Inst. Henri Poincar\'e Probab. Stat.}  45 (2009),  no. 1, 157--174.

\bibitem{PT00}
V. Pipiras, M. Taqqu,
Integration questions related to fractional Brownian motion. 
\textit{Probab. Theory Related Fields} 118 (2000), no. 2, 251--291.

\bibitem{Young36}L.C. Young, an inequality of H\"oder type, connected with Stieltjes integration, \textit{Acta Math.} 67 (1936) 251--282.

\end{thebibliography}
\end{document}